\documentclass[11pt,reqno]{amsart}
\usepackage{graphicx}
\usepackage{amsfonts,amsmath,amssymb}
\def\1#1{\overline{#1}}
\def\2#1{\widetilde{#1}}
\def\3#1{\widehat{#1}}
\def\4#1{\mathbb{#1}}
\def\5#1{\frak{#1}}
\def\6#1{{\mathcal{#1}}}

\def\C{{\4C}}

\newtheorem{thm}{Theorem}[section]
\newtheorem{propos}[thm]{Proposition}
\newtheorem{corol}[thm]{Corollary}

\theoremstyle{definition}
\newtheorem{dfn}[thm]{Definition}
\newtheorem{ex}[thm]{Example}
\newtheorem{rema}[thm]{Remark}

\def\label#1{\label{#1}{\bf (#1)}~}

\newcommand{\im}{\ensuremath{\mbox{\rm Im}\,}}
\newcommand{\re}{\ensuremath{\mbox{\rm Re}\,}}

\newcommand{\theor}[1]{\smallskip  \noindent \bf Theorem #1.\it\,\,}

\newcommand{\CC}[1]{\mathbb{C}^{#1}}

\newcommand{\RR}[1]{\mathbb{R}^{#1}}

\numberwithin{equation}{section}

\title[Convergent normal form and canonical connection]{Convergent normal form and canonical connection for hypersurfaces of finite type in $\CC{2}$}
\author {I. Kossovskiy}
\address{Department of Mathematics, University of Vienna}
\email{ilya.kossovskiy@univie.ac.at}
\author {D. Zaitsev}
\address{School of Mathematics, Trinity College, Dublin}
\email{zaitsev@maths.tcd.ie}

\begin{document}

\date{\today}

\begin{abstract}
We study the holomorphic equivalence problem for finite type hypersurfaces in $\CC{2}$. We discover a geometric condition, which is sufficient for the existence of a natural convergent normal form for a finite type hypersurface. We also provide an explicit construction of such a normal form. As an application, we construct a canonical connection for a large class of finite type hypersurfaces. To the best of our knowledge, this gives the first construction of an invariant connection for Levi-degenerate hypersurfaces in $\CC{2}$.
\end{abstract}

\maketitle \tableofcontents

\section{Introduction}

In their celebrated paper \cite{chern}, Chern and Moser constructed a {\em convergent} normal form for real-analytic hypersurfaces in $\CC{N}$, $N\geq 2$, at Levi-nondegenerate points.  Since then, normal forms have attracted considerable attention, see e.g., Wong \cite{wong}, Ebenfelt \cite{ebenfeltjdg,ebenfeltC3},   Kolar \cite{kolar}, Kolar and
Lamel \cite{kl}, Moser and Webster \cite{mw}, Gong \cite{gong1,gong2}, Huang
and Yin \cite{hy}, Coffman \cite{coffman}, Burcea \cite{bu}, Ezhov and Schmalz \cite{es}, and
Beloshapka \cite{belsing,obzor}. However, the problem of extending Chern-Moser\rq{}s result to {\em Levi-degenerate} points remains widely open. 

In the previous paper \cite{generic}, the authors constructed  {\em convergent} normal form for real-analytic hypersurfaces in $\CC{N}$, $N\geq 2$, at so-called {\em generic Levi degeneracy points} (that is, points where the Levi form is of corank $1$ and its determinant has nonvanishing differential along the Levi kernel). \rm  We refer the reader to the introduction given in that paper, for a general overview and discussion of normal forms and their importance.
We recall that, in the two-dimensional case, a {\em formal}  normal form for finite type hypersurfaces in $\CC{2}$ was previously obtained by Kolar \cite{kolar}. Kolar also showed in \cite{kolardiverg} that his normal form  is {\em divergent} in general. 

In this paper, we are looking for (more general) optimal conditions in $\CC{2}$ when a convergent normal form is possible.  Here we consider  normal forms of the kind
\begin{equation}\label{basic}
v = \sum_{k,l>0}\Phi_{kl}(u) z^k \bar z^l, \quad (z,w)\in\CC{}\times\CC{}, \quad w=u+iv,
\end{equation}
 given by conditions on the coefficient functions $\Phi_{kl}(u)$. (This is motivated by the work of Chern-Moser \cite{chern} and the majority of known constructions in other cases.) In particular, this includes certain nondegeneracy condition imposed on the leading terms $\sum\Phi_{kl}z^k\bar z^l$ with minimal $k+l$. A convergent normal form of the latter kind is only possible when {\em the type of the hypersurface is finite and constant along the line} 
 \begin{equation}\label{Gamma}
\Gamma=\bigl\{z=0,\,v=0\bigr\}=\bigl\{(0,u):u\in\RR{}\bigr\}.
\end{equation}
This motivates the following definition.

\begin{dfn}
Let $M\subset\CC{2}$ be a real-analytic hypersurface and $p\in M$.  We call a smooth real-analytic curve $\gamma\subset M,\,\gamma\ni p$, a \it transverse curve of constant type, \rm if $\gamma$ is transverse to the complex tangent $T^{\CC{}}_q M$ at any point $q\in\gamma$ and the type of $M$ is constant along $\gamma$. 
\end{dfn}

In view of the above discussion, the presence of such curve is {\em a necessary condition} for the existence of the kind of normal form under consideration. We prove here in Theorems 1 and 2 below that this condition is in fact sufficient. 

Recall that {\em the type} of a hypersurface \eqref{basic} at $0$ is the minimum $k+l$ for which $\Phi_{kl}(0)\neq 0$.
The type of a real hypersurface plays an important role, among others, in solving the $\bar\partial$-problem, see e.g. Kohn \cite{kohn}, D\rq{}Angelo \cite{dangelo} and Catlin \cite{catlin2}.   

Once a hypersurface contains a transverse curve of constant type, two fundamentally different cases occur, depending on whether the number of such curves is infinite or finite. These two cases can be invariantly distinguished by the set of points of maximum type: 

\begin{dfn} Let $M\subset\CC{2}$ be a real-analytic hypersurface of finite
type $k\geq 3$ at a point $p\in M$, $U$ a sufficiently small neighborhood of $p$ in $\CC{2}$, and $\Sigma\ni p$ the Levi degeneracy set of $M$. 
We  call the (germ at $p$ of the) set 
$$C=\Bigl\{q\in M\cap U:\,\, \mbox{the\,\,type  of}\,\,  M\,\,\mbox{at\,\,} q\,\,\mbox{equals}\,\,k\,\Bigr\}\subset\Sigma$$ 
 the \it maximal type locus (or set) at $p$. \rm
\end{dfn}

Note that, according to the upper semi-continuity of the type in $\CC{2}$, in a sufficiently small neighborhood of a point $p\in M$ of finite type $k\geq 2$ all other points in $M$ have finite type, not exceeding $k$.  

It will be shown in Section 4 that the maximal type locus is either a single point, or a real-analytic subset  of dimension 1 in $M$,
or a  smooth real-analytic hypersurface in $M$, which coincides with the whole Levi-degeneracy set $\Sigma$. Due to the geometric nature, we have different normalization procedures in the 2- and 1-dimensional cases, as stated in Theorems 1 and 2 respectively.

\subsection{Convergent normal form} We now formulate our first main result. We begin by recalling that a germ of a hypersurface of a finite type $k\geq 3$ can be always represented in suitable holomorphic coordinates as 
\begin{equation}\label{polmodel}
v=P(z,\bar z)+o(|z|^k+|u|),
\end{equation}
where $P(z,\bar z)$ is a non-zero homogeneous polynomial of degree $k$ in $z,\bar z$ without harmonic terms. We expand  \eqref{polmodel} as
$$v=P(z,\bar z)+\sum_{\alpha,\beta\geq 0}
\Phi_{\alpha\beta}(u)z^\alpha\bar z^\beta.$$
We then have $\Phi_{\alpha\beta}(0)=0$ for $\alpha+\beta \leq k$.

\medskip

\theor{1} Let $M\subset\CC{2}$ be a real-analytic hypersurface of
finite type $k\geq 3$ at a point $p\in M$. Assume that the maximal type locus at $p$ has dimension $2$. 
Then there exists a biholomorphic map $F:\, (\CC{2},p)\mapsto (\CC{2},0)$,
which brings $(M,p)$ into a normal form 
$$\Bigl\{ v=P(z,\bar z)
+\sum_{\alpha + \beta \geq k, \, \alpha,\beta>0}
\Phi_{\alpha\beta}(u)z^\alpha\bar z^\beta\Bigr\},$$
 where the polynomial 
 $P(z,\bar z)=\frac{1}{k}\Bigl[(z+\bar z)^k-z^k-\bar z^k\Bigr]$, and $\Phi_{\alpha\beta}$ 
 satisfy 
\begin{equation}\label{strongtube2}
\begin{aligned}
%\Phi_{\alpha 0}=0,\,\alpha\geq 0,\quad 
\Phi_{\alpha 1}=0,\,\alpha \ge k-1,\quad
 \re\Phi_{k,k-1}=\im\Phi_{2k-2,2}=0 .& 
 \end{aligned}
\end{equation}
The normalizing transformation $F$  is uniquely determined by the restriction of its differential $dF_p$ onto the complex tangent $T^{\CC{}}_p M$. Moreover, two germs $(M,p)$ and $(M^*,p^*)$ are biholomorphically equivalent if and only if, for some  (and then for any)  normal forms $(N,0)$ and $(N^*,0)$ of them, there exists a linear map 
\begin{equation}
\label{dilations}\Lambda(z,w)=(\lambda z,\lambda^k
w),\,\lambda\in\RR{}\setminus\{0\},
\end{equation} 
transforming $(N,0)$ into $(N^*,0)$. The Levi degeneracy set $\Sigma$ of $M$, which in this case is a smooth real-analytic totally real surface in $\CC{2}$ transverse to the complex tangent $T^{\CC{}}_p M$  is canonically foliated by distinguished biholomorphically invariant curves, called degenerate chains, where the chain through p is locally given by $\bigl\{z = 0,\,v=0\bigr\}$ in any normal form at p. \rm

\medskip

Theorem 1 addresses the case when the set of transverse curves of constant type through $p$ is infinite. Note that for $k=3$ a finite type hypersurface in $\CC{2}$ has a \it generic Levi degeneracy, \rm and thus automatically satisfies the conditions of Theorem 1. Convergence of a normal form in the latter case was already studied in \cite{generic}, however, for completeness we do not exclude it from our considerations. 

\smallskip

On the other hand, Theorem 2 below addresses the case when the maximal type locus has dimension $1$ and hence the set of transverse curves of constant type through $p$ is non-empty but finite. In the latter case, we call each such  distinguished curve in $M$ again  a {\it degenerate chain}, since it  transforms into the line \eqref{Gamma} in a suitable normal form. \rm The finite collection $\bigl\{\gamma_1,...,\gamma_s\bigr\}$ of degenerate chains in $M$ is an additional biholomorphic invariant of $M$, that is why the normalization procedure in this case is  treated as a normalization of a triple $(M,\gamma_j,p)$
for some choice of an integer $j=1,...,s$. To formulate the normalization result in this case we need to consider the \emph{complex defining equation} 
\begin{equation}\label{complex}
w=\Theta(z,\bar z,\bar w)
\end{equation}
of a hypersurface (see, e.g., \cite{ber}) in addition to the real defining equation 
\begin{equation}\label{real}
v=\Phi(z,\bar z,u).
\end{equation} 
Recall that \eqref{complex} is obtained from \eqref{real} by substituting $u=(w+\bar w)/2,\,v=(w-\bar w)/2i$ and solving for $w$ by using the implicit function theorem (and \eqref{real} is obtained from \eqref{complex} similarly). The function $\Theta$ satisfies the reality condition 
\begin{equation}\label{reality}
\bar\Theta(\bar z,z,\Theta(z,\bar z,\bar w))\equiv \bar w.
\end{equation}
In terms of the complex defining function $\Theta$ the polynomial approximation \eqref{polmodel} reads as
\begin{equation}\label{polmodelc}
w=\bar w + 2iP(z,\bar z)+o(|z|^k+|\bar w|),
\end{equation}
which we expand   as
$$w=\bar w+2iP(z,\bar z)+\sum\nolimits_{\alpha,\beta\geq 0}
\Theta_{\alpha\beta}(\bar w)z^\alpha\bar z^\beta.$$
We then have $\Theta_{\alpha\beta}(0)=0$ for $\alpha+\beta \leq k$. 

\medskip

\theor{2} Let $M\subset\CC{2}$ be a real-analytic hypersurface of
finite type $k\geq 3$ at a point $p\in M$. 
Assume that there exists at least one but finitely many
transverse curves of constant type  in $M$, passing through $p$.
% Let then $\bigl\{\gamma_1,...,\gamma_s\bigr\}$ denotes the finite collection of %degenerate chains in $M$ through $p$. 
Then, for any choice of a transverse curve of constant type (degenerate chain) $\gamma$ through $p$, there exists a biholomorphic map $F:\, (\CC{2},p)\mapsto (\CC{2},0)$,
sending $\gamma$ into the line \eqref{Gamma}
and
 $(M,p)$ into the normal form 
$$ \Bigl\{w=\bar w+2iP(z,\bar z)+\sum_{\alpha,\beta>0,\,\alpha+\beta\geq k}
\Theta_{\alpha\beta}(\bar w)z^\alpha\bar z^\beta\Bigr\},$$
 such that one of the two following cases {\bf (i)} or {\bf (ii)} holds:

\smallskip

\noindent {\bf (i) (circular case)}  The type $k=:2\nu$ is even,  
$$P(z,\bar z)=|z|^k,$$
and the functions $\Theta_{\alpha\beta}$ satisfy
\begin{equation}\label{strongcirc}
%\Phi_{\nu\nu}=0, \quad
%\Phi_{\alpha 0}=0,\,\alpha\geq 0,\quad 
\Theta_{\nu\alpha}=0,\,\alpha \ge \nu,\quad
%\Phi_{\alpha\beta}=0,\,\alpha+\beta\leq k-1,\quad
\im\Theta_{2\nu,2\nu}=\im\Theta_{3\nu,3\nu}=0. 
\end{equation}
The normalizing transformation $F$  is uniquely determined by the restriction of its differential $dF_p$ onto the complex tangent $T^{\CC{}}_p M$ and the restriction of its Hessian $D^2 F_p$ onto the tangent space $T_p\gamma$. In turn, $F$ is unique up to the action of the subgroup 
\begin{equation}
\label{biggroup}\Lambda(z,w)=\left(\frac{\lambda e^{i\theta} z}{(1+rw)^{1/\nu}},
\frac{\lambda^{2\nu}
w}{1+rw}\right),\,\lambda,\theta,r\in\RR{},\,\lambda\neq 0,
\end{equation}  
of the projective group $\mbox{Aut}\,(\mathbb{CP}^2)$.
%Moreover, two germs $(M,p)$ and $(M^*,p^*)$ are biholomorphically equivalent if and only if for some (and then for any) normal forms $(N,0)$ and $(N^*,0)$ of them there exists a linear-fractional map
%\begin{equation}
%\label{biggroup}\Lambda(z,w)=\left(\frac{\lambda e^{i\theta} z}{(1+rw)^{1/\nu}},
%\frac{\lambda^{2\nu}
%w}{1+rw}\right),\,\lambda,\theta,r\in\RR{},\,\lambda\neq 0,
%\end{equation}   
%transforming $(N,0)$ into $(N^*,0)$. 
Moreover, the degenerate chain $\gamma$ through $p$ is unique in the circular case.
% and s given locally by the equation $\{z=0,v=0\}$ in any normal form of $M$ at $p$.

\medskip

\noindent {\bf (ii) (tubular/generic case)} the polynomial $P(z,\bar z)$ has the form
$$P(z,\bar z) = z^\nu\bar z^{k-\nu}+\bar z^\nu z^{k-\nu}+\sum_{\nu+1\leq j\leq k/2}\bigl(a_jz^j\bar z^{k-j}+\bar a_j\bar z^j z^{k-j}\bigr)$$ for some integer $1\leq \nu < k/2$,  and the functions $\Theta_{\alpha\beta}$ satisfy
\begin{equation}\label{stronggeneric}
\begin{aligned}
%\Phi_{\alpha 0}=0,\,\alpha\geq 0,\quad  
\Theta_{\nu\alpha}=0,\,\alpha\geq k-\nu,\quad
% \Phi_{\alpha\beta}=0,\,\alpha+\beta\leq k-1,\quad 
\re\Theta_{2\nu,2k-2\nu}=0 . 
 \end{aligned}
\end{equation}
The normalizing transformation $F$  is uniquely determined by  the initial choice of a degenerate chain $\gamma$ and by the restriction of the differential $dF_p$ onto the complex tangent $T^{\CC{}}_p M$. Moreover, two germs $(M,p)$ and $(M^*,p^*)$ are biholomorphically
equivalent if and only if, for some choice of degenerate chains $\gamma\subset M,\,\gamma^*\subset M^*$ and some (hence any) respective normal forms
$(N,0)$ and $(N^*,0)$, there exists a linear map \eqref{dilations},
transforming $(N,0)$ into $(N^*,0)$.
 \rm

\medskip

\begin{rema} 
Note that for $\nu=1$ the normalization conditions in case (ii) of the theorem can be similarly reformulated in terms of the real defining function $\Phi$. However, for $\nu\geq 2$ the transfer to the real defining function is more complicated. We refer to the work \cite{zaitsevemb} of the second author (see Proposition 3.3 there) where the necessary transfer formulas are provided explicitly. 
\end{rema}

%\begin{rema} 
%We would like to emphasize that, in contrast with the Chern-Moser normal form, our normal form is invariant under the automorphism group of the associated polynomial model. In particular, given any $M$ and $M\rq{}$, any choice of normal forms can be used for their comparison. 
%\end{rema}

\begin{rema} 
It will be shown in the proof of Theorem 2 that the degenerate chain through $p$ in the case {\bf (ii)} is always unique, with the only exception of the so-called {\em tubular case}, when the polynomial $P(z,\bar z)$ has the form 
$$P(z,\bar z)=\frac{1}{k}\Bigl[(z+\bar z)^k-z^k-\bar z^k\Bigr].$$
On the other hand, as Example 5.1 in Section 5 shows, the statement of Theorem 2 in the tubular case  can not be strengthened: there exist hypersurfaces of this class, for which there are more than one (but finitely many) transverse curves of constant type, passing through a finite type point $p\in M$. 
\end{rema}

We call the convergent normal form, provided by Theorems 1 and 2, 
\it the strong normal form. \rm We note that this normal form is different from Kolar\rq{}s {formal}  normal form in \cite{kolar} (see also Remark 2.2 below).

\subsection{Canonical connection along the Levi degeneracy set}

As our second main result, we construct a canonical connection on the Levi degeneracy set in the case when the maximal type locus is 2-dimensional, which gives a different convergent normal form in that case and therefore a solution to the equivalence problem. We emphasize that this is a different kind of connection than the one known in the literature (e.g. \cite{cartan, tanaka, chern, ebenfeltduke, mizner, gami, iz, bes2, cap-schichl, schmslov, eis}). In contrast to situations considered in the cited work, where the connection arises due to certain \lq\lq{}uniformity\rq\rq{}    of the CR-structure, in  our case, already the rank of the Levi form is not constant. In this situation, known methods of constructing connections can not be used. Instead, we use a different, normal form inspired approach to construct a canonical connection only along the Levi degeneracy set. Since the latter is 2-dimensional, this connection turns out to be sufficient to solve the equivalence problem. As a byproduct, taking the normal coordinates of the complexified canonical connection, we obtain a new convergent normal form. 

\medskip

\theor{3} Let $M\subset\CC{2}$ be  a real-analytic hypersurface of finite type $k\geq 3$ at a point $p\in M$. Suppose that the maximal type locus at $p$ has dimension $2$. Then the Levi degeneracy set $\Sigma$ of $M$ is a smooth real-analytic surface, transverse to $T^{\CC{}}_p M$ and totally real in $\CC{2}$. Moreover, one can define  a canonical biholomorphically invariant connection $\nabla$ on $\Sigma$.
The normal coordinates, given by the geodesic flow of  the complexification $\nabla^{\CC{}}$ of
$\nabla$, reduce the CR-equivalence problem between two germs $(M,p)$ and $(M^*,p^*)$ as above to their linear equivalence problem in the normal coordinates. In fact, it suffices to consider linear scalings \eqref{dilations}. 
Thus $\nabla$ provides a solution for the biholomorphic equivalence problem for the above described class of real hypersurfaces. \rm

\medskip

An explicit construction of the connection $\nabla$ is given in Section 6.

\bigskip

\mbox{}

\begin{center}\bf Acknowledgments \end{center}

\smallskip

We are sincerely grateful to Martin Kolar who observed that using the real defining function for the normalization of a hypersurface leads to a singular differential equation for the normalizing transformation in general, and kindly communicated this to us. We would also like to thank the anonimous referee for careful reading of the paper and useful remarks.
 
The first author is supported by the Austrian Science Fund (FWF) Grant M1413N25, and the second author is partially supported by Science Foundation Ireland (SFI) Grant 10/RFP/MTH2878.

\smallskip

\section{Modified formal normal form for finite type hypersurfaces}

In this section we slightly modify Kolar\rq{}s formal normalization procedure
and construct a similar to the one in \cite{kolar} but still different formal normal form for finite type hypersurfaces in $\CC{2}$. In the same way as in \cite{kolar}, it arises from a Chern-Moser type operator associated with a hypersurface. The main difference between the normal form in \cite{kolar} and the one in this section is that \emph{terms of as low as possible degree in $z,\bar z$} are removed from the defining function in our case. This will be used later for the proof of convergence of the normal form in the tubular case, and also for the comparison of the formal normal form with the strong normal form given by Theorems 1 and 2.

\subsection{Polynomial models for finite type hypersurfaces}

Let $M\subset\CC{2}$ be a real-analytic hypersurface of finite type $k\geq 3$ at a point $p\in M$.
As well-known fact, one can choose local holomorphic coordinates  $(z,w)=(z,u+iv)$ near $p$ in such a way that $p$ is the origin, $T_p M=\{v=0\}$, and the defining equation of $M$ looks as 
\begin{equation} \label{tangentmodel}
v=P(z,\bar z)+O(k+1)
\end{equation}
or, in terms of the complex defining function, as
\begin{equation} \label{tangentmodelc}
w=\bar w+2iP(z,\bar z)+O(k+1).
\end{equation}
Here $P(z,\bar z)$ is a non-zero real-valued homogeneous polynomial in $z,\bar z$ of degree $k$ without harmonic terms, and $O(k+1)$ denotes the sum of weighted homogeneous polynomials in $z,\bar z,u$ (resp. $z,\bar z,\bar w$) of weight $\geq k+1$, where the variables are assigned the weights 
$$[z]=[\bar z]=1,\quad [w]=[\bar w]=[u]=[v]=k.$$
The real-algebraic hypersurface 
$$K=\bigl\{v=P(z,\bar z)\bigr\}\sim \bigl\{w=\bar w+2iP(z,\bar z)\bigr\} $$
 is called \it the polynomial model for $M$ at $p$. \rm For the polynomial $P(z,\bar z)$ we use expansion of the form $P(z,\bar z)= \sum_{j=1}^{k-1} a_j z^j \bar z^{k-j},\,a_j\in\CC{},\,\bar a_j=a_{k-j}$. 
As in Kolar~\cite{kolar}, in what follows we use the important

 \bigskip
 
 \noindent\bf Notation. \rm  By $\nu$ we denote the smallest integer $1\leq j \leq k/2$ such that $a_j\neq 0$. 
 
 \bigskip
 
\noindent We then can additionally normalize $P(z,\bar z)$ by the condition $a_\nu=1$ (by means of a scaling in $z$). 

If now $M=\{v=\Phi(z,\bar z,u)\}$ and $M^*=\{v^*=\Phi^*(z^*,\bar
z^*,u^*)\}$ are two hypersurfaces of finite type $k\geq 3$ at the origin, simplified as in \eqref{tangentmodel} and having the same polynomial model $K$,
and $F$ given by
$$z^*=f(z,w),\,w^*=g(z,w)$$ 
is a formal invertible
transformation, transforming $(M,0)$ into $(M^*,0)$, we obtain the
identity
\begin{equation}
\label{tangency} \im g(z,w)|_{w=u+i\Phi(z,\bar z,u)}
=\Phi^*(f(z,w),\overline{f(z,w)},\re g(z,w))|_{w=u+i\Phi(z,\bar
z,u)}.
\end{equation}
In terms of the complex defining functions $\Theta,\Theta^*$ of $M,M^*$ respectively \eqref{tangency} looks as 
\begin{equation}\label{tangencyc}  
g(z,w)|_{w=\Theta(z,\bar z,\bar w)}
=\Theta^*(f(z,w),\bar f(\bar z,\bar w),\bar g(\bar z,\bar w))|_{w=\Theta(z,\bar z,\bar w)}
\end{equation}
(note that \eqref{tangencyc} is an identity of two power series in the free variables $z,\bar z,\bar w$). We call either of the identities \eqref{tangency},\eqref{tangencyc} \emph{the basic identity} for the transformation $F$. 

For the formal power series
$f(z,w)\in\CC{}[[z,w]]$ we consider the formal
expansion
$$f(z,w)=\sum_{m\geq 1}f_m(z,w),$$
 where each $f_m(z,w)$ is a
weighted homogeneous polynomial of weight $m$ (with respect to the above choice of weights), and similarly for $g(z,w)$. For the defining function $\Phi(z,\bar z,u)$ we use expansion of the form 
$$\Phi(z,\bar z,u)=P(z,\bar z)+\sum\limits_{j,l\geq 0} \Phi_{jl}(u) z^j\bar z^l$$ 
(note that each term in the latter sum has weight $\geq k+1$).
We also denote for each $m\geq 2$ 
$$\Phi^{(m)}(z,\bar z,u):=\sum\limits_{j+l=m} \Phi_{jl}(u) z^j\bar z^l.$$
Similarly, for the complex defining function $\Theta(z,\bar z,\bar w)$ we use expansion of the form 
$$\Theta(z,\bar z,\bar w)=\bar w+2iP(z,\bar z)+\sum\limits_{j,l\geq 0} \Theta_{jl}(\bar w) z^j\bar z^l.$$ 
We also use for $m\geq 2$ the notation 
$$\Theta^{(m)}(z,\bar z,\bar w):=\sum\limits_{j+l=m} \Theta_{jl}(\bar w) z^j\bar z^l.$$
We then consider three cases, in each of which the normal form is significantly different.  

\smallskip

\noindent {\bf Tubular case.} In this case $\nu=1$, and the polynomial $P(z,\bar z)$ has the form $$P(z,\bar z)=\frac{1}{k}\Bigl[(z+\bar z)^k-z^k-\bar z^k\Bigr].$$
The polynomial model $K$ in  this case is polynomially equivalent to a tubular real hypersurface. In addition, $K$ is invariant under the Lie group \eqref{dilations}
of transformations, preserving the origin. An inspection of low weight terms in \eqref{tangency} (see \cite{kolar}) shows that the initial transformation $F=(f,g)$ can be uniquely decomposed as $F=\tilde F\circ\Lambda$, where $\Lambda$ is as in \eqref{dilations} and the weighted components of the new mapping $\tilde F$ satisfy
\begin{equation}
\label{normalmaptube} f_1=z,\quad g_1=\dots=g_{k-1}=0,\quad g_k=w.
\end{equation}
We now say that a formal hypersurface $(M,0)$ with a tubular polynomial model \it is in normal form, \rm if its defining function $\Phi(z,\bar z,u)$ satisfies 

\begin{equation}
\begin{aligned}\label{normalformtube}
\Phi_{\alpha 0}=0,\,\alpha\geq 0,\quad \Phi_{\alpha 1}=0,\,\alpha\geq k-1,\\
\re\Phi_{k-2,1}=  
 \re \Phi_{k,k-1}=\im\Phi_{2k-2,2}=0.
 \end{aligned}
\end{equation}

\smallskip

\noindent {\bf Circular case.} In this case $k$ is even, $\nu=k/2$, and the polynomial $P(z,\bar z)$ has the form $$P(z,\bar z)=z^\nu\bar z^\nu.$$
The polynomial model $K$ in  this case can be mapped into the quadric $\mathcal Q=\bigl\{v=|z|^2\bigr\}$ by means of the polynomial map $z\mapsto z^\nu,\,w\mapsto w$, and hence is spherical at Levi nondegenerate points. It is also invariant under the Lie group \eqref{biggroup}
of transformations, preserving the origin. An inspection of low weight terms in \eqref{tangency} (see \cite{kolar}) shows that the initial transformation $F=(f,g)$ can be uniquely decomposed as $F=\tilde F\circ\Lambda$, where $\Lambda$ is as in \eqref{biggroup} and the weighted components of the new mapping $\tilde F$ satisfy
\begin{equation}
\label{normalmapcirc} f_1=z,\quad g_1=\dots=g_{k-1}=0,\quad g_k=w,\quad \re g_{ww}(0,0)=0.
\end{equation}
We now say that a formal hypersurface $(M,0)$ with a circular polynomial model \it is in normal form, \rm if its complex defining function $\Theta(z,\bar z,\bar w)$ satisfies 

\begin{multline}\label{normalformcirc}
\im\Theta_{00}=0,\,\Theta_{0\alpha}=0,\,\alpha\geq 1,\quad 
\Theta_{\nu\alpha }=0,\,\alpha\ge \nu,\\
 \Theta_{\nu,\nu-1}=\im\Theta_{\nu,\nu}=\im\Theta_{2\nu,2\nu}=\im\Theta_{3\nu,3\nu}=0 
\end{multline}
(note that, unlike the real defining function case, for the symmetric terms $\Theta_{jj}(u)$ neither $\re\Theta_{jj}(u)$ nor $\im\Theta_{jj}(u)$ vanish in general!).
\smallskip

\noindent {\bf Generic case.} In this case the polynomial model $K$ is neither tubular nor circular, and we call it (following Kolar ~\cite{kolar}) \it generic \rm in what follows. In the same way as in the tubular case, $K$ is invariant under the Lie group \eqref{dilations}
of transformations, preserving the origin, and an inspection of low weight terms in \eqref{tangency}  shows that the initial transformation $F=(f,g)$ can be uniquely decomposed as $F=\tilde F\circ\Lambda$, where $\Lambda$ is as in \eqref{dilations} and the weighted components of the new mapping $\tilde F$ satisfy \eqref{normalmaptube}. To describe the normal form in the generic case, we need to introduce (see also \cite{kolar}) a natural Hermitian form on each space $\mathbb H_{m}$ of homogeneous (but not necessarily real-valued) polynomials in $z,\bar z$ of degree $m\geq 2$. Namely, if 
$Q=\sum_{j=0}^{m} a_j z^j \bar z^{m-j} \in \mathbb H_{m}, \,  R=\sum_{j=0}^{m} b_j z^j \bar z^{m-j} \in \mathbb H_{m},\, a_j,b_j\in\CC{}$, we define
\begin{equation}\label{scalar}
(Q,R):=\sum\limits_{j=1}^{m-1} a_j\bar b_j.
\end{equation}
This Hermitian form \eqref{scalar} is degenerate, however, it is nondegenerate on the subspace  
$$\mathcal L_k\subset\mathcal H_k$$
 consisting of polynomials without harmonic terms (thus \eqref{scalar} is a scalar product on each $\mathcal L_k$).
We now say that a formal hypersurface $(M,0)$ with a generic polynomial model \it is in normal form, \rm if its complex defining function $\Theta(z,\bar z,\bar w)$ satisfies 

\begin{equation}\label{normalformgeneric}
\begin{aligned}
\im\Theta_{00}=0,\,\Theta_{0\alpha }=0,\,\alpha\geq 1,\quad \Theta_{\nu\alpha}=0,\,\alpha\geq k-\nu,\\
\re\Theta_{2\nu,2k-2\nu}=0,\quad \bigl(\Theta^{(k-1)},P_z\bigr)=0.
\end{aligned}
\end{equation}

\subsection{Modified Kolar\rq{}s normal form}   We now formulate our formal normalization result
 on the modified Kolar's normal form.

\begin{propos} Let $M\subset\CC{2}$ be a real-analytic hypersurface of finite type $k\geq 3$ at a point $p$, simplified as above, and  $K=\bigl\{v=P(z,\bar z)\bigr\}$  its polynomial model at $p$.  

\smallskip

(i) Let the polynomial model $K$ be of tubular type. Then for any choice of a non-zero real parameter $\lambda$ there exists a unique formal invertible transformation $F=(f,g)$, bringing $(M,0)$ into a normal form \eqref{normalformtube} and satisfying $f_z(0,0)=\lambda$.  

\smallskip

(ii) Let the polynomial model $K$ be of circular type. Then for any choice of a non-zero real parameter $\lambda$ and real parameters $\theta,\rho$ there exists a unique formal invertible transformation $F=(f,g)$, bringing $(M,0)$ into a normal form \eqref{normalformcirc} and satisfying $f_z(0,0)=\lambda e^{i\theta},\,\re g_{ww}(0,0)=\rho$.  

\smallskip

(iii) Let the polynomial model $K$ be of generic type. Then for any choice of a non-zero real parameter $\lambda$ there exists a unique formal invertible transformation $F=(f,g)$, bringing $(M,0)$ into a normal form \eqref{normalformgeneric} and satisfying $f_z(0,0)=\lambda$.

\end{propos}

\begin{rema}
We emphasize that this normal form is merely formal to begin with.
It does coincide with the normal form in Theorem~1, implying its convergence in that case. However, it does not, in general, coincide with the normal form in Theorem~2. The reason being that, in general, this formal normal form
may not transfer the distinguished degenerate chain $\gamma$
 into the line \eqref{Gamma}. Instead, in Theorem~2 
 we construct another convergent normal form
 trasfering the degenerate chain into the line \eqref{Gamma}
 and having fewer conditions on the defining equation of the hypersurface.
\end{rema}

The strategy of the proof of Proposition 2.1 is similar to that used in Kolar~\cite{kolar}. As follows from the discussion above, it is sufficient to prove the claim of Proposition 2.1 for the case of a map $F$, satisfying \eqref{normalmaptube} in the tubular and generic cases, and \eqref{normalmapcirc} in the circular case, so that we assume $F$ to be normalized accordingly. We also use for the defining functions $\Phi,\Theta$ of $M$ expansion of the form 
$$\Phi(z,\bar z,u)=P(z,\bar z)+\sum\limits_{m\geq k+1}\Phi_m(z,\bar z,u),\quad \Theta(z,\bar z,\bar w)=\bar w+2iP(z,\bar z)+\sum\limits_{m\geq k+1}\Theta_m(z,\bar z,\bar w),$$
where all  $\Phi_m(z,\bar z,u),\Theta_m(z,\bar z,\bar w)$ are weighted homogeneous polynomials of weight $m$. 
We then consider \eqref{tangency} as an
infinite series of weighted homogeneous equations, which can be
written for any fixed weight $m\geq k+1$ as
\begin{equation}
\label{CM}\re\left(ig_{m}+P_zf_{m-k+1}\right)\left|_{w=u+iP(z,\bar z)}\right.=\Phi^*_m-\Phi_m+\cdot\cdot\cdot,
\end{equation}
where dots stands for a polynomial in $z,\bar z,u$ and
 $f_{j-k+1},g_{j}$ with $j<m$ and their
derivatives in $u$ (here
$f_{j-k+1}=f_{j-k+1}(z,u),\,g_j=g_{j}(z,u),\,\Phi_j=\Phi_j(z,\bar
z,u))$. Similarly, for the identity \eqref{tangencyc} we get
\begin{equation}
\label{CMc} g_{m}(z,w)-\bar g_m(\bar z,\bar w)-2iP_zf_{m-k+1}(z,w)-2iP_{\bar z}\bar f_{m-k+1}(\bar z,\bar w)\left|_{w=\bar w+2iP(z,\bar z)}\right.=\Theta^*_m-\Theta_m+\cdot\cdot\cdot.
\end{equation}

Let us denote by $\mathcal F$ the space of formal real-valued
power series $\Phi(z,\bar z,u)=\sum_{m\geq k+1}\Phi_m(z,\bar z,u)$, and by $\mathcal N\subset\mathcal F$ the subspace of formal series, satisfying  the normalization conditions  
\eqref{normalformtube}. Similarly, we use the notation $\mathcal F^c$ for power series of the form $\sum_{m\geq k+1}\Theta_m(z,\bar z,\bar w)$, and $\mathcal N^c\subset\mathcal F^c$ for the subspace of formal series, satisfying  the normalization conditions \eqref{normalformcirc} or \eqref{normalformgeneric}, depending on the polynomial model under consideration. Denote also by $\mathcal G$ the spaces of pairs of the form $\{(f,g)-\mbox{Id}\}$, where $f,g$ are
 formal power series without constant term, satisfying
\eqref{normalmaptube} or \eqref{normalmapcirc}, also depending on the polynomial model under consideration.
%Let us decompose the space $\mathcal F$ of formal real-valued
%power series $\Phi(z,\bar z,u)=\sum_{m\geq 4}\Phi_m(z,\bar z,u)$,
%satisfying \eqref{simplified}, into the direct sum $\mathcal
%R\oplus\mathcal N$, where $\mathcal N\subset\mathcal F$ is the
%subspace of series satisfying the normalization conditions
%\eqref{normalform}, and $\mathcal R\subset\mathcal F$ is the range
%of the $\RR{}$-linear operator $P:\,\mathcal F\mapsto \mathcal F$,
%$P^2=P$, defined by
%\begin{eqnarray*}P(\Phi):=\sum\limits_{k\geq
%0}(\Phi_{k0}(u)z^k+\Phi_{0k}(u)\bar z^k)+ \sum\limits_{k\geq
%2}(\Phi_{k1}(u)z^k\bar z+\Phi_{1k}(u)z\bar z^k)+\\+ (i\im
%\Phi_{32}(u)z^3\bar z^2+i\im\Phi_{23}(u)z^2\bar z^3)+(\re
%\Phi_{42}(u)z^4\bar z^2+\re \Phi_{24}(u)z^2\bar z^4\end{eqnarray*}
%($P$ is the projection in $\mathcal F$ onto the subspace $\mathcal
%R$).
In view of \eqref{CM}, in order to prove Proposition 2.1, it is sufficient to prove
the following proposition.

\begin{propos}
 Let $M$ be a real-analytic hypersurface of finite type $k\geq 3$ at a point $p$, simplified as above,   $K=\bigl\{v=P(z,\bar z)\bigr\}$  its polynomial model at $p$,  $L$ denotes the linear operator 
$$ L(f,g):=\re\left(ig+P_z f\right)\left|_{w=u+iP(z,\bar z)}\right.,$$  
and
$L^c$  the linear operator 
$$ L^c(f,g):=g(z,w)-\bar g(\bar z,\bar w)-2iP_zf(z,w)-2iP_{\bar z}\bar f(\bar z,\bar w)\left|_{w=\bar w+2iP(z,\bar z)}\right.$$ 
Then in the tubular case
  we have
the direct sum decomposition $\mathcal F=L(\mathcal
G)\oplus\mathcal N$, and  in either of the  circular or generic cases the direct sum decomposition $\mathcal F^c=L^c(\mathcal
G)\oplus\mathcal N^c$.
\end{propos}

Indeed, consider, for example, the tubular case. Then it follows from Proposition 2.3  that if $m\geq k+1$ is an
integer and all $f_{j-k+1},g_{j},\Phi^*_j$ with $j<m$ are already
determined, then one can uniquely choose the collection
$(f_{m-k+1},g_{m},\Phi^*_m)$ in such a way that \eqref{CM} is
satisfied and $\Phi^*_m\in\mathcal N$. This implies the existence
and uniqueness of the desired normalized mapping $(f,g)$ and the
normalized right-hand side $\Phi^*(z,\bar z,u)$. For the other cases the argument is analogues.

\begin{proof}[Proof of Proposition 2.3] 

 \mbox{}
 
 \medskip

\noindent \bf Tubular case. \rm The statement of the proposition is equivalent to the fact that an
equation 
\begin{equation}\label{Lfg}
L(f,g)=\Psi(z,\bar z,u),\quad(f,g)\in\mathcal G
\end{equation}
 in $(f,g)$
has a unique solution, modulo $\mathcal N$ in the right-hand side,
for any fixed $\Psi\in\mathcal F$. To simplify the calculations,
we rescale \eqref{Lfg} as
\begin{equation}
\label{cm-eq} 2L(f,g)=\Psi(z,\bar z,u),\quad (f,g)\in\mathcal
G,\,\Psi\in\mathcal F,
\end{equation}
which we solve in $f,g$. We use expansions of the form
$$f(z,u+iP(z,\bar z))=f(z,u)+f_u(z,u)iP+\frac{1}{2}f_{uu}(z,u)i^2P^2+\cdot\cdot\cdot.$$ Substituting into \eqref{cm-eq} we get
the equation
\begin{multline}
\label{cm-eq1} i\left(g(z,u)+g_u(z,u)iP+\frac{1}{2}g_{uu}(z,u)i^2P^2+\cdot\cdot\cdot\right)+\\
 +P_z\left(f(z,u)+f_u(z,u)iP+\frac{1}{2}f_{uu}(z,u)i^2P^2+\cdot\cdot\cdot\right)+\\
+ \,\{\mbox{complex conjugate terms}\}=\Psi(z,\bar z,u).
\end{multline}
For the sequel of the proof, we expand $f(z,w)$ as $f=\sum\limits_{k\geq 0}f_k(w)z^k,$ and
similarly for $g$. 

   Collecting in \eqref{cm-eq1} terms of
bi-degrees $(\alpha,0),\,\alpha\geq 0$ in $z,\bar z$, we get
\begin{equation}
\label{k0} ig(z,u)-i\bar g_0(u)+z^{k-1}\bar f_0(u)=\sum_{\alpha\geq 0}\Psi_{\alpha 0}(u)z^\alpha.
\end{equation}
The equation \eqref{k0} enables us to determine uniquely the functions $g_\alpha(u),\,\alpha\neq 0, k-1$, as well as the function $\im g_0(u)$, in such a way that the
conditions $\Psi_{\alpha 0}=0,\,\alpha\neq k-1$ are achieved. In addition, we have 
\begin{equation}\label{k-1,0}
ig_{k-1}+\bar f_0=\Psi_{k-1,0}.
\end{equation} 
Gathering then terms of bi-degrees $(\alpha,1)$ with $\alpha\geq 1$, we get
\begin{equation}
\label{k1}
-g_uz^{k-1}-\bar g_0\rq{}z^{k-1}+(k-1)z^{k-2}f+z^{k-1}\bar f_1+(k-1)z^{k-2}\bar f_0-iz^{2k-2}\bar f_0\rq{}=\sum_{\alpha\geq 1}\Psi_{\alpha 1}(u)z^\alpha.
\end{equation}
The equation \eqref{k1} enables us to determine uniquely all $f_\alpha,\,\alpha\neq 0,1,k,$ as well as 
$\re f_0$, in order to achieve the conditions $\Psi_{\alpha 1}=0,\,\alpha\geq k,\,\alpha\neq 2k-2$ and $\re\Psi_{k-2,1}=0$. Thus we uniquely determine from \eqref{k-1,0} $\im g_{k-1}$ and achieve $\re\Psi_{k-1,0}=0$. In addition, we have the conditions 
\begin{equation}\label{k-1,1}
-2\re g_0\rq{}+(k-1)f_1+\bar f_1=\Psi_{k-1,1}
\end{equation}
and
\begin{equation}\label{2k-2,1}
-g_{k-1}\rq{}+(k-1)f_k-i\bar f_0\rq{}=\Psi_{2k-2,1}.
\end{equation}
By considering the imaginary part in \eqref{2k-2,1} we uniquely determine $\im f_k$ and achieve the condition $\im\Phi_{2k-1,1}=0$.  

In what follows we use the notation 
$$\{\mbox{\bf pdt}\}$$
 for a linear differential expression, depending on previously determined terms. Considering then in \eqref{cm-eq1} terms of bi-degree $(2k-2,2)$, we obtain 
\begin{equation}\label{2k-2,2} 
\{\mbox{\bf pdt}\}+i(k-1)f_1\rq{}-i\bar f_1\rq{}=\Psi_{2k-2,2}.
\end{equation}
Recall that, in view of \eqref{normalmaptube}, $f_1\rq{}(0)=g_0\rq{}(0)=1$. Hence the imaginary part of the equation \eqref{2k-2,2}, read together with the equation \eqref{k-1,1}, enables us to determine $\re g_0,f_1$ uniquely and achieve the conditions $\Psi_{k-1,1}=\im\Psi_{2k-2,2}=0$.
Finally, we consider in \eqref{cm-eq1} terms of bi-degree $(k,k-1)$, we obtain 
\begin{equation}\label{k,k-1}
-g_{k-1}\rq{}+f_k-ic\bar f_0\rq{}=\Psi_{k,k-1},
\end{equation}
where $c=c(k)$ is a positive integer. Then, by considering the imaginary part in \eqref{k-1,0} and the real parts in \eqref{2k-2,1} and \eqref{k,k-1}, we get a $3\times 3$ real linear system for $\im f_0\rq{},\re g_{k-1}\rq{},\re f_k$, which is nondegenerate for any positive $c$ and $k\geq 3$. Thus we determine 
$\im f_0,\re g_{k-1},\re f_k$ uniquely (thanks to $f_0(0)=0$) and achieve $\im\Psi_{k-1,0}=\re\Psi_{2k-2,1}=\Psi_{k,k-1}=0$. 

All the terms, concerned in the normalization conditions \eqref{normalformtube}, have been considered, and this proves the proposition in the tubular case.

\smallskip

\noindent \bf Circular case. \rm  For the circular and the generic cases we have to replace the equation \eqref{cm-eq} by
\begin{equation}\label{Lfg1}
L^c(f,g)=\Psi(z,\bar z,\bar w),\quad(f,g)\in\mathcal G
\end{equation}
 and prove that it  
has a unique solution in $(f,g)$, modulo $\mathcal N^c$ in the right-hand side,
for any fixed $\Psi\in\mathcal F^c$. We also replace
the equation \eqref{cm-eq1} by
\begin{multline} \label{cm-eq2}
 -\bar g(\bar z,\bar w)+g(z,\bar w)+2ig_w(z,\bar w)P+2i^2g_{ww}(z,\bar w)P^2+\cdot\cdot\cdot-\\
 -2i\Bigl[P_{\bar z}\bar f(\bar z,\bar w) -P_zf(z,\bar w)-2if_w(z,\bar w)P_zP-2i^2f_{ww}(z,\bar w)P_zP^2-\cdot\cdot\cdot\Bigr]
=\Psi(z,\bar z,\bar w).
\end{multline}

Collecting in \eqref{cm-eq2} terms of
bi-degrees $(0,\alpha),\,\alpha\geq 0$ in $z,\bar z$, we get
\begin{equation}
\label{k0ii} -\bar g(\bar z,\bar w)+ g_0(\bar w)=\sum_{\alpha\geq 0}\Psi_{0 \alpha}(\bar w)z^\alpha.
\end{equation}
The equation \eqref{k0ii} enables us to determine uniquely the functions $g_\alpha(\bar w),\,\alpha>0$, as well as the function $ \im g_0(u)$, in such a way that the
conditions $\Psi_{0\alpha}=0,\,\alpha\geq 1$ and $\im\Psi_{00}=0$ are achieved. 

Gathering then terms of bi-degrees $(\nu,\alpha)$ with $\alpha\geq 1$, we get
\begin{equation}
\label{k1ii}
2ig_0\rq{}\bar z^{\nu}-2i\nu\bar z^{\nu-1} \bar f(\bar z,\bar w)-2i\nu \bar z^\nu f_1=\sum_{\alpha\geq 1}\Psi_{\nu\alpha }(\bar w)\bar z^\alpha.
\end{equation}
The equation \eqref{k1ii} enables us to determine uniquely all $f_\alpha,\,\alpha\neq 1$, in order to achieve the conditions $\Psi_{\nu\alpha}=0,\,\alpha\geq \nu+1,$ and $\Psi_{\nu-1,\nu}=0$ (recall that we have $\nu=k-\nu$ in the circular case).  In addition, we have the condition
\begin{equation}\label{kk}
-2i\nu( f_1+\bar f_1)+2i g_0\rq{}=\Psi_{\nu\nu}.
\end{equation}

Considering then in \eqref{cm-eq2} terms of bi-degrees $(2\nu,2\nu)$ and $(3\nu,3\nu)$, we get, respectively,
\begin{equation}\label{2k2k}
-2 g_0\rq{}\rq{}-4\nu f_1\rq{}=\Psi_{2\nu,2\nu}
\end{equation}
and
\begin{equation}\label{3k3k}
-\frac{4i}{3} g_0\rq{}\rq{}\rq{}-2i\nu f_1\rq{}\rq{}=\Psi_{3\nu,3\nu}.
\end{equation}
Recall that, thanks to \eqref{normalmapcirc}, we have $g_0(0)=g_0\rq{}(0)=\im f_1\rq{}(0)=\re g_0\rq{}\rq{}(0)=0$ and $\re f_1\rq{}(0)=1$. Hence we determine $\im f_1$ uniquely from \eqref{2k2k} and achieve $\im\Psi_{2\nu,2\nu}=0$, and also determine $\re g_0,\re f_1$ uniquely from \eqref{kk},\eqref{3k3k} and achieve $\im\Psi_{\nu,\nu}=\im\Psi_{3\nu,3\nu}=0$. 

All the terms, concerned in the normalization conditions \eqref{normalformcirc}, have been considered, and this proves the proposition in the circular case.

\smallskip

\noindent \bf Generic case. \rm Collecting in \eqref{cm-eq2} terms of
bi-degrees $(0,\alpha),\,\alpha\geq 0,\,\alpha\neq k-1$ in $z,\bar z$, we get
\begin{equation}
\label{k0iii}-\bar g(\bar z,\bar w)+ g_0(\bar w)=\sum_{\alpha\neq k-1}\Psi_{0\alpha}(\bar w)\bar z^\alpha.
\end{equation}
 The equation \eqref{k0iii} enables us to determine uniquely the functions $g_\alpha(u),\,\alpha\neq 0;k-1$, as well as the function $\im g_0(u)$, in such a way that the
conditions $\Psi_{\alpha 0}=0,\,\alpha\neq k-1; 0$,  and $\im\Psi_{00}=0$ are all achieved. 

Let us introduce the unique polynomial $Q\in\mathcal L_k$ ($\mathcal L_k$ being the space of polynomials without harmonic terms as before), which coincides with $P_z$ modulo harmonic terms. We have $Q,\bar Q\in\mathcal L_k$. We then consider in \eqref{cm-eq2} 
all terms of bi-degrees $z^\alpha\bar z^\beta$, $\alpha+\beta=k-1$, and apply to both sides the multiplication by $Q$ with respect to the Hermitian form \eqref{scalar}. We get: 
\begin{equation}\label{k-1}
-2if_0(Q,Q)-2i\bar f_0(\bar Q,Q)=\bigl(\Psi^{(k-1)},Q\bigr).
\end{equation}
Since $(\cdot,\cdot)$ is nondegenerate on $\mathcal L_k$, by the Cauchy inequality we have 
$$|(\bar Q,Q)|\leq \sqrt{(Q,Q)}\cdot\sqrt{(\bar{Q},\bar{Q})}=(Q,Q),$$ 
and the equality achieved if and only if $Q$ is a scalar multiple of $\bar Q$. It is easy to see that in the latter case the polynomial model $K=\bigl\{v=P(z,\bar z)\bigr\}$ is tubular, which is a contradiction. 
Thus $(Q,Q)>|(\bar Q,Q)|$, and the equation \eqref{k-1} enables us to determine $f_0$ uniquely and
achieve the condition $(\Psi^{(k-1)},P_z)=0$. Now, by considering in \eqref{cm-eq2} terms of bi-degree $(0,k-1)$, we obtain 
$$-\bar g_{k-1}+\{\mbox{\bf pdt}\}=\Psi_{0,k-1},$$
and we determine from here $g_{k-1}$ uniquely and achieve the condition $\Psi_{0,k-1}=0$.

Gathering then terms of bi-degrees $(\nu,\alpha)$ with $\alpha\geq k-\nu$, we get
\begin{equation}
\label{k1iii}
2i g_0\rq{}\cdot \bar z^{k-\nu}-2i(k-\nu)\bar z^{k-\nu-1}(\bar f-\bar f_0)-2i\nu \bar z^{k-\nu} f_1+\{\mbox{\bf pdt}\}=\sum_{\alpha\geq k-\nu}\Psi_{\nu\alpha}(u)z^\alpha.
\end{equation}
The equation \eqref{k1iii} enables us to determine uniquely all $f_\alpha,\,\alpha>1$, and achieve the conditions $\Psi_{\alpha \nu}=0,\,\alpha> k-\nu$. In addition, we have the condition 
\begin{equation}\label{f1}
2i g_0\rq{}-2i(k-\nu)\bar f_1-2i\nu\bar f_1=\Psi_{\nu,k-\nu}.
\end{equation}
Note that, since the polynomial model $K$ is not circular, we have $\nu<k/2$, so that we determine 
$\im f_1$ uniquely by considering the imaginary part of the equation \eqref{f1} and achieve $\re\Psi_{\nu,k-\nu}=0$. Finally, we consider in \eqref{cm-eq2} terms of bi-degree $(2\nu,2k-2\nu)$, and obtain 
\begin{equation}\label{g0} 
-2g_0\rq{}\rq{}+\{\mbox{\bf pdt}\}-4(k-\nu)f_1\rq{}=\Psi_{2\nu,2k-2\nu}.
\end{equation}
Recall that, in view of \eqref{normalmaptube}, $f_1\rq{}(0)=g_0\rq{}(0)=1$. Hence the real part of the equation \eqref{g0}, read together with the equation \eqref{f1}, enables us to determine $\re g_0,\re f_1$ uniquely and achieve the conditions $\Psi_{\nu,k-\nu}=\re\Psi_{2\nu,2k-2\nu}=0$.

All the terms, concerned in the normalization conditions \eqref{normalformcirc}, have been considered, and this completes the proof of the proposition.
\end{proof}

\begin{corol}
Suppose that $(N,0)$ and $(\tilde N,0)$ are two different formal normal forms of a
fixed germ $(M,p)$  of a real-analytic hypersurface of finite type $k\geq 3$ at the point $p$. Suppose, in addition, that the polynomial model for $M$ at $p$ is either tubular or generic. Then there
exists a linear transformation $\Lambda$, as in \eqref{dilations},
which maps $(N,0)$ into $(\tilde N,0)$.
\end{corol}

\section{Special normal form for hypersurfaces of finite type}

We now aim to construct a complete  normal form
 under the action of a {\em special subgroup} of the group $\widehat{\mbox{Aut}}(\CC{2},0)$ of formal biholomorphisms of $(\CC{2},0)$. Namely, we consider the subgroup 
\begin{equation}\label{subgroup}
\Bigl\{F=(f,g)\in \widehat{\mbox{Aut}}(\CC{2},0):\quad f(0,w)=0,\quad \im g(0,u)=0\Bigr\}.
\end{equation}
This subgroup has a natural geometric interpretation as the subgroup of 
all formal biholomorphisms, preserving the distinguished curve \eqref{Gamma}.
Thus, a normalization procedure under the action of the group \eqref{subgroup} for real-analytic hypersurfaces,  already containing the curve \eqref{Gamma}, can be interpreted as a normalization of a triple 
$$(M,\gamma,p), \quad p\in\gamma\subset M,$$ 
where $M$ is a real-analytic hypersurface of finite type $k$ at the distinguished point $p$, and $\gamma\subset M$ is a distinguished curve through $p$, transverse to the complex tangent $T^{\CC{}}_p M$. 

We now describe the special normal form in each of the cases. Note that if a hypersurface $M\ni 0$ contains the curve \eqref{Gamma}, then its defining function satisfies $\Phi_{00}(u)=0,\,\Theta_{00}(\bar w)=0$.

First, if a (formal) hypersurface $M$ is of tubular type, we say that \it $M$ is in a special normal form, \rm if its defining function satisfies
\begin{equation}
\begin{aligned}\label{specialtube}
\Phi_{\alpha 0}=0,\,\alpha\geq 0,\quad \Phi_{\alpha 1}=0,\,\alpha\geq k,\quad \im\Phi_{2k-2,2}=0.
 \end{aligned}
\end{equation}
Next, if a (formal) hypersurface $M$ is of circular type, we say that \it $M$ is in a special normal form, \rm if its defining function satisfies
\begin{equation}\label{specialcirc}
\Theta_{0\alpha}=0,\,\alpha\geq 0,\quad 
\Theta_{\nu\alpha}=0,\,\alpha\ge \nu,\quad
\im\Theta_{\nu,\nu}=\im\Theta_{2\nu,2\nu}=\im\Theta_{3\nu,3\nu}=0 , \quad
 \nu = k/2.
\end{equation}
Finally, if a (formal) hypersurface $M$ is of generic type, we say that \it $M$ is in a special normal form, \rm if its defining function satisfies
\begin{equation}\label{specialgeneric}
\begin{aligned}
\Theta_{0\alpha }=0,\,\alpha\geq 0,\quad \Theta_{\nu,\alpha}=0,\,\alpha\geq k-\nu, \quad
\re\Theta_{2\nu,2k-2\nu}=0,
\end{aligned}
\end{equation}
where $\nu$ is as in Theorem~2.
We now formulate our special normalization statement.

\begin{propos} Let $M$ be a real-analytic hypersurface of finite type $k\geq 3$ at a point $p$, simplified as in \eqref{tangentmodel} and containing the curve \eqref{Gamma}, and  $K$  its polynomial model at $p$.  

\smallskip

(i) Let the polynomial model $K$ be of tubular type. Then for any choice of a non-zero real parameter $\lambda$ there exists a unique formal invertible transformation $F$, as in \eqref{subgroup}, bringing $(M,0)$ into a special normal form \eqref{specialtube} and satisfying $f_z(0,0)=\lambda$.  

\smallskip

(ii) Let the polynomial model $K$ be of circular type. Then for any choice of a non-zero real parameter $\lambda$ and real parameters $\theta,\rho$ there exists a unique formal invertible transformation $F=(f,g)$,  as in \eqref{subgroup}, bringing $(M,0)$ into a special normal form \eqref{specialcirc} and satisfying $f_z(0,0)=\lambda e^{i\theta},\,\re g_{ww}(0,0)=\rho$.  

\smallskip

(iii) Let the polynomial model $K$ be of generic type. Then for any choice of a non-zero real parameter $\lambda$ there exists a unique formal invertible transformation $F=(f,g)$, as in \eqref{subgroup}, bringing $(M,0)$ into a special normal form \eqref{specialgeneric} and satisfying $f_z(0,0)=\lambda$.  

\end{propos}

\begin{proof}
The proof is completely analogues to the one in Section 2.2, and we leave the details to the reader. The minor differences are as follows. 

\smallskip

\noindent(i) One has to replace the space $\mathcal G$ by its subspace $\mathcal G_1$ determined by  
$$\im g_0(u)=0,\,f_0(u)=0,$$
and then prove the direct decompositions 
$$ \mathcal F=L(\mathcal
G_1)\oplus\mathcal N, \quad \mathcal F^c=L^c(\mathcal
G_1)\oplus\mathcal N^c.$$
Because of the additional constraints on the map, less terms can be finally eliminated in the defining functions $\Phi,\Theta$ of $M$, that is why the special normal form appears to be simpler than the one in Section 2 or the one in \cite{kolar}. 

\smallskip

\noindent
(ii) The spaces $\mathcal F$ and $\mathcal F^c$ should be modified by adding the conditions $\Phi_{00}=0$  and $\Theta_{00}=0$, respectively.
\end{proof}

\begin{corol}
Suppose that $(N,0)$ and $(\tilde N,0)$ are two different special normal forms of a
fixed triple $(M,\gamma,p)$, when $M$ is a real-analytic hypersurface of finite type $k\geq 3$ at the point $p$ and $\gamma\subset M$, $\gamma\ni p$ is a curve, transverse to the complex tangent $T^{\CC{}}_p M$. Suppose, in addition, that the polynomial model for $M$ at $p$ is either tubular or generic. Then there
exists a linear transformation $\Lambda$, as in \eqref{dilations},
which maps $(N,0)$ into $(\tilde N,0)$.
\end{corol}

\section{Transverse curves of constant type and degenerate chains}

\subsection{Structure of the maximal type locus} Let $M\subset\CC{2}$ be a real-analytic hypersurface of finite type $k\geq 3$ at a point $p\in M$. Recall that the subset of points in a neighborhood of $p$ in $M$, at which the type equals $k$, is called the maximal type locus at $p$. \rm  Before proceeding with the proof of Theorems 1 and 2, we need the following two structure results for the maximal type locus of a finite type hypersurface.

\begin{propos} Let $M\subset\CC{2}$ be a real-analytic hypersurface of
finite type $k\geq 3$ at a point $p\in M$. Then the following alternative holds.

\medskip

\it (I)   Either the maximal type locus at $p$ is a single point,

\medskip

or

\medskip

\it (II)   the maximal type locus at $p$ is  a real-analytic subset of $M$ of dimension $1$,

\medskip

or

\medskip

\it (III)  the maximal type locus is a smooth  real-analytic
hypersurface in $M$, which coincides with the whole Levi degeneracy set $\Sigma$. 
\end{propos}

\medskip

\begin{proof}
 Let $\Delta(q)$ denotes the Levi
form of $M$ (which can be considered as a scalar function on $M$ in the case of $\CC{2}$), and $\Sigma$ the Levi degeneracy set of $M$. We have 
$$\Sigma=\{q\in M:\,\Delta(q)=0\}.$$ 
Note that, as follows from \eqref{tangentmodel}, the type of $M$ at a point $q$ equals $l\geq 3$ if and only if 
\begin{equation}\label{typel}
\Delta(q)=d\Delta(q)|_{T^{\CC{}}_q}=...=d^{l-3}\Delta(q)|_{T^{\CC{}}_q}=0, \quad \mbox{while} \quad d^{l-2}\Delta(q)|_{T^{\CC{}}_q}\neq 0.
\end{equation} 
 Thus  the maximal type locus $C\subset\Sigma$ is contained in the
 real-analytic set
\begin{equation}\label{bigtype}
E=\bigl\{q\in\Sigma:\,\Delta(q)=d\Delta(q)|_{T^{\CC{}}_q}=...=d^{k-3}\Delta(q)|_{T^{\CC{}}_q}=0\bigr\},
\end{equation}
describing the set of points in $M$ of type $\geq k$. However, since the type $k$ is locally maximal for points in $M$ near $p$, we have $C=E$ (locally near $p$). If $E$ has dimension $0$ or $1$, then the assertion in (I) or (II), respectively, holds. 
Otherwise $E=C$ has dimension $2$, and the same holds for $\Sigma$. Since, in addition, $d^{k-2}\Delta(p)|_{T^{\CC{}}_q}\neq 0$, we conclude (after intersecting $M$ with a sufficiently small neighborhood $U$ of $p$, if necessary) that $C$ is smooth at $p$. To prove that $C=\Sigma$, we notice that if $\tilde C\neq C$ is another component of the real-analytic set $\Sigma$, passing through $p$ (which has dimension $1$ or $2$),  then we can find a point $q\in C$ near $p$, which is not contained in any other component of $\Sigma$ rather than $C$. Clearly, the characterizing property \eqref{typel} can not hold 
at $p$ and $q$ simultaneously with the same $l$, in view of the existence of an additional component in $\tilde C\subset\Sigma$, while both $p$ and $q$ lie in the maximal type locus $C$, which is a contradiction. This proves the assertion in (III). 
\end{proof} 

From here we can deduce

 \begin{propos} Let $M\subset\CC{2}$ be a real-analytic hypersurface of
finite type $k\geq 3$ at a point $p\in M$, and assume that $M$ contains a transverse curve of constant type through $p$. Then the following alternative holds.

\medskip

\noindent (I)  Either the Levi degeneracy set
$\Sigma$ contains only finitely many transverse curves of constant type, passing through $p$,

\medskip

or

\medskip

\it (II)  the Levi degeneracy set
$\Sigma\subset M$ contains infinitely many transverse curves of constant type, passing through $p$, and in this case the maximal type locus at $p$ coincides with the whole Levi degeneracy set $\Sigma$, while the latter is a smooth real-analytic hypersurface in $M$, transverse to the complex tangent $T^{\CC{}}_p M$ and totally real in $\CC{2}$. Moreover, the polynomial model for $M$ in the latter case is tubular at any point $q\in \Sigma$.
\end{propos}

 \begin{proof}
 After applying Proposition 4.1, it remains to prove that in case (II) the polynomial model for $M$
is tubular at any point $q\in \Sigma$. Without loss of generality we may assume $q=p$. We choose local holomorphic coordinates near $p$ in such a way that $p$ is the origin, $T_0 M=\{v=0\}$, and the totally real manifold $\Sigma$ is given by $\{\re z=0,\,v=0\}$. If we denote the
polynomial model for $M$ at $p$ by $P(z,\bar z)$, then the Levi form of $M$ near $p$ looks as
\begin{equation}\label{Determinant}
\Delta(z,\bar z,u)=P_{z\bar z}+O(|u|)+O(|z|^{k-1})
 \end{equation}
 Since the set $\Sigma$ coincides with the maximal type locus $C$, given by \eqref{bigtype}, the right hand side in \eqref{Determinant} vanishes identically for $\re z=0$, so that $P_{z\bar z}$ vanishes identically for $\re z=0$. Similarly, all the derivatives $P_{z^i\bar z^j},\,i,j\geq 1,\,i+j\leq k-1$ vanish identically for $\re z=0$, while at least one of the derivatives $P_{z^i\bar z^j},\,i,j\geq 1,\,i+j=k$ is non-zero at the origin. This proves that, up to harmonic terms, $P(z,\bar z)=\lambda(\re z)^k,\,\lambda\neq 0$, as required.
 
 \end{proof}
 
 In accordance with Proposition 4.2, for  a real-analytic hypersurface $M\subset\CC{2}$ of finite type $k\geq 3$ at the point $0\in M$, simplified as in \eqref{tangentmodel}, we consider four cases, in each of which the convergent normalization procedure is significantly different.
 
 \smallskip
 
 \noindent \bf Case G. \rm In this case the tangent model $K$ at $0$ is generic, and there exist
 finitely many transverse curves of constant type through $0$ (all of which are contained in a real-analytic set of dimension $1$).
 
 \smallskip
 
 \noindent \bf Case C. \rm In this case the tangent model $K$ at $0$ is circular, and there exist
 finitely many transverse curves of constant type through $0$ (all of which are contained in a real-analytic set of dimension $1$).
 
 \smallskip
 
 \noindent \bf Case T1. \rm In this case the tangent model $K$ at $0$ is tubular, and there exist
 finitely many transverse curves of constant type through $0$ (all of which are contained in a real-analytic set of dimension $1$).
 
  \smallskip
 
 \noindent \bf Case T2. \rm In this case the tangent model $K$ at $0$ is tubular, and there exist
 infinitely many transverse curves of constant type through $0$; the Levi degeneracy set $\Sigma\subset M$ coincides with the maximal type locus $C$ at $p$ and is a totally real surface in $\CC{2}$, transverse to the complex tangent $T^{\CC{}}_0 M$. Note that, as follows from Proposition 4.2, $M$ is of class T2 at any other point $p\in C$ in this case.
 
 \smallskip
 
 \subsection{Degenerate chains} We next introduce in each of the cases distinguished transverse curves of constant type, passing 
 through $0$ and lying in the Levi degeneracy set $\Sigma$ of $M$. We call them \it degenerate chains. \rm 
 
 If $M$ is of one of the classes G, C or T1, we simply call each of the transverse curves of constant type, passing though $0$, a \it degenerate chain. \rm We have at most finitely many degenerate chains, passing through $0$ in these cases.  
 
 If, otherwise, $M$ is of class T2, the construction of degenerate chains is more involved.  We first introduce a canonical pair of foliations in the Levi degeneracy set $\Sigma$. Both ones are defined by the use of appropriate slope (line) fields in the smooth real-analytic manifold $\Sigma$, which we integrate then. 
 
 To define the first slope field,  we choose a point $p\in \Sigma$ and coordinates $(z,w)$
vanishing at $p$, where $M$ takes the form \eqref{tangentmodel}. 
Clearly, these coordinates can be chosen polynomial with
coefficients depending analytically on $p$.
% and make a polynomial coordinate
%change, bringing $p$ to the origin and $M$ to the form \eqref{simplified} (that we denote by $M_S$).
Let $N$ denotes a (formal) normal form  \eqref{normalformtube} of $M$
at $p$, $F$ a formal transformation, mapping $(M,p)$ onto $(N,0)$,
and $e:=(0,1)\in\CC{2}$. We then define a slope at $p$ as follows:
$$l(p):=\mbox{span}_{\RR{}}\left\{(dF|_p)^{-1}(e)\right\}\subset T_{p}\Sigma.$$
We have to prove the inclusion $l(p)\subset T_{p}\Sigma$ and the fact that $l(p)$ is analytic in $p$. First, it follows from Corollary 2.4 that the definition of $l$ is
independent of the choice of normal form. Moreover, the desired
slope can be also defined without using formal transformations.
Indeed, it follows from the normal form construction that, as soon
as the initial weighted polynomials
$\{\Phi^*_j,f_{j-k+1},g_j,\,k+1\leq j\leq m\}$ for some $m\geq k+1$ have
been determined, they do not change after further normalization of
terms of higher weight. Hence, solving the equations \eqref{CM}
for $m$ sufficiently large (namely, for all $m\leq 2k-1$), we uniquely determine
%the collection
%$\{f_2,f_3,g_4,g_5\}$ and thus determine
$dF|_p$. Note that, if a hypersurface of class T2 satisfies \eqref{tangentmodel}, then the tangent space to its Levi degeneracy set at the origin is given by $\bigl\{\re z=0,\,v=0\bigr\}$. Hence the latter tangent space for a hypersurface, normalized up to weight $2k-1$, contains the vector $e$, and it follows from the invariance of $\Sigma$ that $l(p)\subset T_{p}\Sigma$.
It is also not difficult to see from here that the constructed slope
field is analytic (i.e., depends analytically on a point
$p\in \Sigma$). Indeed, the explicit construction in the beginning
of Section 2 shows that each fixed weighted polynomial $\Phi_m$,
 depends on $p$ analytically (this can be
verified from the parameter version of the implicit function
theorem). Hence polynomials $f_m$ and $g_m$ depends on $p$
analytically, as it is obtained by solving a system of linear
equations with a fixed nondegenerate matrix in the left-hand side and
right-hand side analytic in $p$ (the latter fact can be seen from
the proof of Proposition 2.3). We immediately conclude that
$dF|_p$ depends on $p$ analytically, and so does $l(p)$.
We then integrate the analytic slope field $l(p)$ and obtain a
canonical (non-singular) real-analytic foliation $\mathcal T$ of
the Levi degeneracy set $\Sigma$. We call each of the leaves of the foliation $\mathcal T$ \it a
degenerate chain. \rm 
 Each degenerate chain
$\gamma\subset \Sigma$ at each point $p\in\gamma$ is transverse to
the complex tangent $T^{\CC{}}_pM$.

The second canonical foliation in $\Sigma$ corresponds to the
slope field 
$$c(p):=T^{\CC{}}_p M\cap T_p\Sigma.$$ 
Integrating
$c(p)$ we obtain another canonical foliation $\mathcal S$ in
$\Sigma$, which is everywhere tangent to $T^{\C}M$.
%The leaves of the foliation can be described explicitly as intersections $\{Q_p\cap\Sigma,\,p\in\Sigma\}$.
Both foliations $\mathcal T$ and $\mathcal S$ are transverse to
each other and are biholomorphic invariants of $(M,0)$. We call
them respectively {\em transverse} and {\em tangent} foliations.

Degenerate chains play a crucial role in the convergent normalization procedure, described in the next section.

 \section{Convergent normalization to the strong normal form}
 
 \subsection{Construction of the normalizing transformation} 
We proceed now with the proof of the existence of a normalizing transformation in Theorems 1 and 2. As before, $M\subset\CC{2}$ denotes a real-analytic hypersurface of finite type $k\geq 3$ at the point $0\in M$. We assume $(M,0)$ to be simplified, as in \eqref{tangentmodel} or \eqref{tangentmodelc}. In each of the cases the normalizing transformation, satisfying \eqref{normalmaptube} (cases G,T1 and T2) or \eqref{normalmapcirc} (case C) and bringing $(M,0)$ into one of the normal forms \eqref{strongcirc},\eqref{strongtube2},\eqref{stronggeneric}, is a finite composition of biholomorphic transformation, each of which is a normalization of certain geometric data, associated with a finite type hypersurface. We describe below the transformations and their geometric meaning. For the set-up of the theory of Segre varieties see, e.g., \cite{ber}. The first two steps are the same in all of cases, while the final steps are significantly different in each case. 

\bigskip

\noindent{\it Normalization of a chain.} We choose any degenerate chain $\gamma\subset\Sigma$, passing through $0$,
and perform a biholomorphic transformation, transferring $\gamma$ into the curve
$$\Gamma=\{z=0,\,\, v=0\}.$$
Since such a transformation has the form $z^*=z+O(|w|),\,w^*=w+O(|w|^2)$, it preserves the forms \eqref{tangentmodel},\eqref{tangentmodelc}.  Now the new hypersurface $M$ satisfies \eqref{tangentmodel},\eqref{tangentmodelc}, contains $\Gamma$ and has the constant type $k$ at each point $q\in\Gamma$. In addition, $\Phi_{00}(u)=0,\,\Theta_{00}(\bar w)=0$.

\medskip

\noindent \it Normalization of Segre varieties along a chain. \rm
The next step in the normalization procedure is the elimination of
$(0,\alpha)$ terms in the expansion of $\Phi$, which is sometimes
addressed as transfer to \it normal coordinates \rm (see
\cite{ber}).   Geometrically, this step can be interpreted as
straightening of the Segre varieties $Q_p,\,p\in\Gamma$. In the latter case, for the complex defining function we immediately get $\Theta^*(z,0,u)=u$, so that the $(0,\alpha)$ terms are removed from $\Theta$ as well. In addition, it is easy to see from the connection between $\Phi$ and $\Theta$ that we have also $\Theta^*_{\alpha 0}=0,\,\alpha\geq 0$

  According
to \cite{chern}, \cite{ber}, we perform the unique transformation
of the form
\begin{equation}
\label{kill-k0} z^*=z,\quad w^*=w+g(z,w), \quad g(0,w)=0,
\end{equation}
which maps $M$ into a hypersurface with $\Phi^*_{0\alpha}=0,\,\alpha\geq 0$ (and hence $\Theta^*_{0\alpha}=\Theta^*_{\alpha 0}=0,\,\alpha\geq 0$). 
This transformation preserves the curve $\Gamma$ and the forms \eqref{tangentmodel},\eqref{tangentmodelc}. 
Segre varieties of points $p=(0,\eta)\in\Gamma$ look all now as $\{w=\eta\}$, and
for any $p\in\Gamma$ we now have $T^{\CC{}}_p M=\{w=0\}$. 

Presence of the transverse curve $\Gamma\subset\Sigma$ of constant type and the fact that all the complex tangents along $\Gamma$ have the form $\{w=0\}$ imply strong consequences for the defining function $\Phi$ of $M$. Namely, the approximations  \eqref{tangentmodel},\eqref{tangentmodelc} become invariant under shifts $w\mapsto w+u_0$, so that 
for the new hypersurface
\begin{equation}\label{strong}
\Phi^*_{\alpha 0}=\Phi^*_{0\alpha}=\Theta^*_{0\alpha}= \Theta^*_{\alpha 0}=0,\,\alpha\geq 0,\quad \Phi^*_{\alpha\beta}=\Theta^*_{\alpha\beta}=0,\,\alpha,\beta\geq 0,\, \alpha+\beta\leq k-1.
\end{equation}

\medskip

From now on we consider several cases.

\bigskip

\begin{center}\bf Case T1 and case G with $\nu=1$.\end{center}

\bigskip

In this case we deal with the real defining function $\Phi$ only.

\medskip

\noindent \it Normalization of the Segre map. \rm  Our goal is to bring a hypersurface, obtained in the previous step, to such a form that it satisfies, in addition, the condition $$\Phi^*_{\alpha 1}=\Phi^*_{1 \alpha }=0,\,\alpha\geq k-1.$$

 This step can be
interpreted as a normalization of the Segre map of $M$ near $0$. 
Recall that the 
{\em Segre map} of $M$ assigns to each point $p$ in a neighborhood of the origin in $\CC{2}$ a (fixed) sufficiently
  high jet of its Segre variety $Q_p$ at the intersection point of $Q_p$  with $z=0$. In this step we may take the $1$-jet  and normalize the Segre map as
$$
(\xi,\eta)\mapsto\left(\bar\eta,2i\bar\xi^{k-1}\right) .$$

We first achieve the condition $\Phi_{k-1,1}(u)=0$. 
Choose a real-valued analytic function  $\lambda(u)$,
for $u$ in a neighborhood of $0$,
 such that $\lambda(0)=1$ and 
 $$1+\Phi_{k-1,1}(u)=(\lambda(u))^k$$
  (recall that $\Phi_{k-1,1}(0)=0$). We perform the
biholomorphic change
$$z^*=z\lambda(w),\,w^*=w.$$
 Since $w^*=w$, we get
\begin{equation}\label{nochange}
\Phi(z,\bar z,u)=\Phi^*(z^*,\bar z^*,u) \,\, \mbox{for}\,\, (z,w)\in M.
\end{equation} 
By comparing in \eqref{nochange} terms of uniform degree $<k$ in $z,\bar z$, we
conclude that $\Phi^*$ also satisfies \eqref{strong}.
Further, by comparison of  $(k-1,1)$ terms in $z,\bar z$ in \eqref{nochange}, we obtain 
$$1+\Phi_{k-1,1}(u)=\left(1+\Phi^*_{k-1,1}(u)\right)(\lambda(u))^k,$$
so that $M$ is mapped into a hypersurface satisfying $\Phi^*_{k-1,1}=0$, as required.

 Next, for a hypersurface, satisfying $\Phi_{k-1,1}=0$, we perform a transformation
$$z=z^*+f(z^*,w^*),\quad w=w^*, \quad f(z^*,w^*)=O(|z^*|^{2})$$
in order to remove the $(\alpha,\nu)$-terms with $\alpha\geq k$. 
Clearly, the target hypersurface $M^*$ satisfies the previously achieved normalization conditions. Let us put $f=:z^*h,\,h=O(|z^*|)$.
 Using \eqref{strong} and \eqref{nochange}, we compute, by comparing in \eqref{nochange} all $(\alpha,1)$ terms with $\alpha\geq k$: 
$$\sum\limits_{\alpha\geq k}\Phi^{*}_{\alpha 1}(u^*)(z^*)^\alpha \overline{z^*}=
\sum\limits_{\alpha\geq k}\Phi_{\alpha 1}(u^*)(z^*)^\alpha\overline{z^*}+(z^*)^{k-1}\overline{z^*}
\Big((1+h)^{k-1}-1\Bigr).$$
Dividing the latter identity by $(z^*)^{k-1}\overline{z^*}$, we obtain a suitable $h$ with $h=O(|z^*|)$ by the implicit function theorem.
%It is easy to see that
%$\Gamma$ is also preserved.
%The new defining function is
%$$\Im w - \Phi(z+f, \bar z+\bar f, u),$$
%which is $O_{22}$ when restricted to $\Pi$. The fact that
%$z^*=z+O(|z|^2),w^*=w$ shows $\Phi^*_{11}=0$ and $\Phi_{21}^*(0)=1$.

We end up with a hypersurface satisfying 
 \begin{equation}\label{strong1}
\Phi^*_{\alpha 0}=0,\,\alpha\geq 0,\quad \Phi^*_{\alpha\beta}=0,\,\alpha,\beta\geq 0,\, 
\alpha+\beta\leq k-1,\quad \Phi^*_{\alpha 1}=0,\alpha\geq k-1.
\end{equation}

\medskip

\noindent \it  Fixing a parameterization on a degenerate chain. \rm Since $M$ already satisfies \eqref{strong1},
it remains to achieve the condition $\im\Phi_{2k-2,2}=0$, as the transfer to the complex defining equation given then all the normalization conditions in \eqref{stronggeneric}.  We do so by means of a \it gauge transformation \rm 
\begin{equation}\label{gauge}
z\mapsto z f(w),\quad w\mapsto g(w),\quad f(0)\neq 0,\,g(0)=0,\,g\rq{}(0)\neq 0.
\end{equation}
This transformation can be interpreted as a choice of a parameterization on the degenerate chain $\Gamma$.
We choose 
$$f(w):=(g'(w))^{1/k},\quad g(w)\in\RR{}\{w\},\quad g(0)=0,\,g\rq{}(0)=1.$$ 
It is easy to see that \eqref{strong1} is preserved under such transformation.
Let us denote by $\mathcal S$ the space of convergent series $\Psi(z,\bar z,u)$, satisfying  \eqref{stronggeneric}. Then, using again expansions of kind
$h(u+iv)=h(u)+ih'(u)v+\dots,$ we compute
\begin{gather*}
v^*=g'(u)v\quad\mbox{mod}\,\mathcal S,\\
z^*(\bar z^*)^{k-1}=z\bar z^{k-1}g\rq{}(u)
\left(1+i\frac{g''(u)}{g'(u)}v)\right)^{1/k}\left(1-i\frac{g''(u)}{g'(u)}v)\right)^{(k-1)/k}\quad\mbox{mod}\,\mathcal S =\\
=g'(u)z\bar z^{k-1}+i\frac{2-k}{k}g''(u)z^{2}\bar
z^{2k-2}\quad\mbox{mod}\,\mathcal S.
\end{gather*}
Similarly, for $1<\alpha\leq k/2$ we get:
\begin{gather*}
(z^*)^\alpha(\bar z^*)^{k-\alpha}=g'(u)z^\alpha\bar z^{k-\alpha}\quad\mbox{mod}\,\mathcal S,
\end{gather*}
thus
\begin{gather*}
v^*-P(z^*,\bar z^*)=g'(u)(v-P(z,\bar z))+i\frac{2-k}{k}g''(u)\bigl(z^{2k-2}\bar z^2-z^{2}\bar
z^{2k-2}\bigr)\quad\mbox{mod}\,\mathcal S.
\end{gather*}
We conclude that 
$$\im \Phi^*_{2k-2,2}=g'\im \Phi_{2k-2,2}+\frac{2-k}{k}g'',$$
so that the condition $\im \Phi^*_{2k-2,2}(u)=0$ leads to a second
order \it nonsingular \rm ODE. Solving it with the initial
condition $g'(0)=1$, we finally obtain a hypersurface of the form
\eqref{stronggeneric}, as required. 

At the same time, we have  brought $M$ to a special normal form \eqref{specialgeneric} or \eqref{specialtube}, respectively, corresponding to the triple $(M,\gamma,p)$ for the chosen degenerate chain $\gamma$.  Next, it is
not difficult to see, performing similar calculations, that the
gauge transformation chosen to achieve $\im \Phi_{2k-2,2}=0$ must have
the above form $z^*=z(g'(w))^{1/k},\,w^*=g(w)$ and hence is unique
up to the choice of the real parameter $g'(0)$, corresponding to
the action of the group \eqref{dilations}. (To see that the gauge transformation has the above form, one has to first compare in the basic identity 
$$\im g(w)=\Phi^*(zf(w),\bar z\bar f(\bar w),\re g(w))|_{w=u+i\Phi(z,\bar z,u)}$$
 terms independent of $z,\bar z$ and obtain $\im g(u)=0$, and second compare terms of kind $z\bar z^{k-1} u^l,\,l\geq 0$ and obtain $g'(u)=f(u)(\bar f(u))^{k-1}$, which gives the desired identity $g'(u)=(f(u))^k$ in view of $k\geq 3$).
   Thus, remarkably,
we can canonically, up to the action of the group of dilations
\eqref{dilations}, choose a parametrization on each degenerate
chain.\rm

\bigskip

\begin{center}\bf Case G with $\nu>1$.\end{center}

\bigskip

In this case we aim normalize the first and the $(\nu+1)$-th components of the Segre map as 
\begin{equation}
\label{Segre}p=(\xi,\eta)\mapsto
\left(\bar\eta,*,...,*,2i\bar\xi^{k-\nu}\right) .
\end{equation}
and canonically, up to the action of the group of dilations
\eqref{dilations}, choose a parametrization on the degenerate chain chosen above (we refer here for \cite{elz09} for some general facts on the normalization of the Segre map). This finally gives the strong normal form \eqref{stronggeneric}. However, the desired  normalizing transformation  can {\it not} be anymore decoupled as in the case $\nu=1$. Instead, we achieve the normalization conditions in the following three steps.

\medskip

\noindent \it Step I. \rm  Our goal here is to bring a hypersurface obtained in the previous step to such a form that it satisfies, in addition, the conditions 
$$\Theta^*_{\nu,\alpha}=0,\quad k-\nu\leq \alpha< 2k-\nu.$$
We first achieve the condition $\Theta_{k-\nu,\nu}=0$. For that, arguing identically to the tubular case, we find an appropriate  (complex-valued) analytic function  $\lambda(u)$ such that $\lambda(0)=1$ and the
biholomorphic change
$$z^*=z\lambda(w),\,w^*=w$$
maps $M$ to a hypersurface satisfying, in addition to \eqref{strong}, the condition $\Phi^*_{\nu,k-\nu}=0$. Now it is not difficult to check, from the implicit function procedure connecting $\Phi$ and $\Theta$, that we also have   $\Theta^*_{\nu,k-\nu}=0$.

 Second, for the hypersurface $M$, obtained in the
previous step, we remove all $(\nu,\alpha)$-terms with $k-\nu<\alpha< 2k-\nu$ in $\Phi$  by a
sequence of $(k-1)$ transformation $F_2,...,F_{k}$, where each $F_j$ has the form
$$(z^*,w^*)=(z+f_j(w)z^j,w).$$
 Let us analyze the hypersurface $M^*$, obtained after a transformation $F_2$. We aim to remove the $(\nu,k-\nu+1)$-terms from $M$ by means of this transformation. The basic identity looks as 
 $$\Theta(z,\bar z,\bar w)=\Theta^*\bigl(z+f_2(\Theta(z,\bar z,\bar w))z^2,\bar z+\bar f_2(\bar w)\bar z^2,\bar w\bigr).$$
 By comparing in the basic identity terms of uniform degree $\leq k$ in $z,\bar z$, we see that $M^*$ satisfies \eqref{strong} and also  $\Theta^{*}_{\nu,k-\nu}=0$.  Moreover, we have
 \begin{equation}\label{k-degree}
 \Theta^*_{\alpha\beta}(\bar w)=\Theta_{\alpha\beta}(\bar w)\quad\mbox{for}\quad\alpha+\beta=k.
 \end{equation}
  Then, using  \eqref{k-degree} and the fact that  all terms in $\Theta,\Theta^*$ have uniform degree $\geq k$ in $z,\bar z$, we compute, by comparing in the basic identity the $(\nu,k-\nu+1)$ terms:
 $$\Theta_{\nu,k-\nu+1}(\bar w)=2i(k-\nu)\bar  f_2(\bar w)+\Theta^*_{\nu,k-\nu+1}(\bar w)+(\nu-1)\Theta_{\nu-1,k-\nu+1}(\bar w)f_2(\bar w).$$
 Since $\Theta_{\nu-1,k-\nu+1}(0)=0$ (by the definition of $\nu$), the latter identity determines the holomorphic near the origin function $f_2$ uniquely. Proceeding with similar arguments, we uniquely determine holomorphic functions $f_3,...,f_{k}$ such that the composition $F_ k\circ\dots\circ F_2$ maps $M$ into a hypersurfaces satisfying 
 \begin{equation}\label{strong2}
\Theta^*_{\alpha 0}=\Theta^*_{0\alpha}=0,\,\alpha\geq 0,\quad \Theta^*_{\alpha\beta}=0,\,
\alpha+\beta\leq k-1,\quad \Theta^*_{\nu\alpha}=0,\,k-\nu\leq\alpha<2k-\nu.
\end{equation}

\medskip

\noindent \it Step II. \rm In this step we perform a transformation
$$z^*=\lambda(w)z+s(w)z^{k+1},\quad w^*=g(w),\quad g(w)\in\RR{}\{w\},\,g(0)=0,\,g\rq{}(0)\neq 0,\quad s(w)\in\CC{}\{w\}$$
with
$$\lambda(w):=\bigl(g\rq{}(w)\bigr)^{1/k}.$$
We aim to achieve the two conditions
$$\Theta^*_{\nu,2k-\nu}=\re\Theta^*_{2\nu,2k-2\nu}=0.$$
The basic identity looks as 
$$g(\Theta(z,\bar z,\bar w))=\Theta^*\bigl(\lambda(\Theta(z,\bar z,\bar w))z+s(\Theta(z,\bar z,\bar w))z^{k+1},\lambda(\bar w)\bar z+\bar s(\bar w)\bar z^{k+1},g(\bar w)\bigr).$$
By considering in the basic identity terms of uniform degree $\leq 2k-1$ in $z,\bar z$ we conclude that the new hypersurface also satisfies \eqref{strong2}. Moreover, we have, in view of the choice of $\lambda$:
\begin{equation}\label{lowdegree}
\Theta_{\alpha\beta}(\bar w)=\Theta^*_{\alpha\beta}(g(\bar w)),\quad \alpha+\beta=k.
\end{equation}
 We then consider the $(\nu,2k-\nu)$ and the $(2\nu,2k-2\nu)$ terms in the basic identity. By using \eqref{strong2} and \eqref{lowdegree} we see that the condition $\Theta^*_{\nu,2k-\nu}=0$ amounts to
 $$2ig\rq{}(\bar w)\Theta_{\nu,2k-\nu}+C(\bar w)g\rq{}\rq{}(\bar w)=2i(k-\nu)(\lambda(\bar w))^{k-1} \bar s(\bar w)+R(\bar w)(\lambda(\bar w))^k\frac{\lambda\rq{}(\bar w)}{\lambda(\bar w)}.$$
 Here $C(\bar w),R(\bar w)\in\CC{}\{\bar w\}$ are some precise functions \emph{vanishing at the origin}  and depending polynomially on $\Theta_{\alpha\beta},\,\alpha+\beta=k,\,\alpha\neq\nu$, exact form of which is of no interest to us. Similarly, by using \eqref{strong2} and \eqref{lowdegree} we see that the condition $\re\Theta^*_{2\nu,2k-2\nu}=0$ amounts to
 \begin{gather*}
 g\rq{}(u)\re\bigl[2i\Theta_{2\nu,2k-2\nu}(u)\bigr]-2g\rq{}\rq{}(u)+T(u)g\rq{}\rq{}(u)=\\
 =-4\nu(\lambda(u))^k \frac{\lambda\rq{}(u)}{\lambda(u)}+A(u)(\lambda(u))^k\frac{\lambda\rq{}(u)}{\lambda(u)}+(\lambda(u))^{k-1}\re\bigl[B(u)\bar s(u)\bigr].
 \end{gather*}
 Here $A(u),T(u)\in\RR{}\{u\},B(u)\in\CC{}\{u\}$ are again some precise functions vanishing at the origin  and depending polynomially on $\Theta_{\alpha\beta},\,\alpha+\beta=k,\,\alpha\neq\nu$, exact form of which is of no interest to us.
  
   We obtain a system
 \begin{equation}\label{degree2k}
 \begin{aligned}
&2ig\rq{}(\bar w)\Theta_{\nu,2k-\nu}+C(\bar w)g\rq{}\rq{}(\bar w)=2i(k-\nu)g\rq{}(\bar w) \frac{\bar s(\bar w)}{\lambda(\bar w)}+\frac{1}{k}R(\bar w)g\rq{}\rq{}(\bar w),\\
&g\rq{}(u)\re\bigl[2i\Theta_{2\nu,2k-2\nu}(u)\bigr]-2g\rq{}\rq{}(u)+T(u)g\rq{}\rq{}(u)=
-\frac{4\nu}{k}g\rq{}\rq{}(u)+\\
+&\frac{1}{k}A(u)g\rq{}\rq{}(u)+g\rq{}(u)\re\bigl[B(u) \frac{\bar s(u)}{\lambda(u)}\bigr].
 \end{aligned}
 \end{equation}
 We now solve the first equation in \eqref{degree2k} for $\frac{\bar s(\bar w)}{\lambda(\bar w)}$ and substitute into the second. Since $\frac{4\nu}{k}<2$ in the generic case, this gives us a second order \emph{nonsingular} analytic ODE for $g(\bar w)$, which we solve uniquely with some initial conditions $g(0)=0,\,g\rq{}(0)\neq 0$.

 \medskip

\noindent \it Step III. \rm Finally, for the hypersurface, obtained  in the previous step, we perform a transformation
$$z^*=z+f(z,w),\quad w^*=w,\quad f(z,w)=O(|z|^{k+2})$$
by which we aim to achieve the conditions $\Theta^*_{\nu\alpha}=0,\,\alpha>2k-\nu$.
The basic identity looks as
$$
\Theta(z,\bar z,\bar w)=\Theta^*\bigl(z+f(z,\Theta(z,\bar z,\bar w)),\bar z+\bar f(\bar z,\bar w),\bar w\bigr).$$
We immediately obtain from here that all terms in $\Theta$ of uniform degree $\leq 2k$ in $z,\bar z$ remain unchanged after such transformation, so that all the previously achieved normalization conditions are preserved. Collecting then in the basic identity all terms of kind $(\nu,\alpha),$ we see that the requirement $\Theta^*_{\nu\alpha}=0,\,\alpha>2k-\nu$ amounts to
$$ 2i\bar z^{k-\nu}+\sum_{\alpha>2k-\nu}\Theta_{\nu\alpha}(\bar w)\bar z^\alpha=2i(\bar z+\bar f(\bar z,\bar w))^{k-\nu}.$$
The latter equation enables us to determine $f$ with the desired properties uniquely. It is not difficult to see that this argument is reversible. Thus the new hypersurface $M$ satisfies all the normalization conditions in \eqref{stronggeneric}.
%It is easy to see that
%$\Gamma$ is also preserved.
%The new defining function is
%$$\Im w - \Phi(z+f, \bar z+\bar f, u),$$
%which is $O_{22}$ when restricted to $\Pi$. The fact that

Note now that we have already brought $M$ to a special normal form \eqref{specialgeneric}, corresponding to the triple $(M,\gamma,p)$ for the chosen degenerate chain $\gamma$.    Thus, in the same way as in the case $\nu=1$,
we can canonically, up to the action of the group of dilations
\eqref{dilations}, choose a parametrization on each degenerate
chain (the parameter on the degenerate chain is fixed by solving a second order ODE in the above Step II).

\bigskip

\begin{center}\bf Case C.\end{center}

\bigskip

The construction of the normalizing transformation in this case is very similar to that in the generic case with $\nu>1$, that is why we leave the details to the reader and only provide below a scheme of the normalization procedure for a hypersurface, already satisfying \eqref{strong}. 

\medskip

\noindent \it Step I. \rm  We achieve the conditions 
$$\Theta^*_{\nu\alpha}=0,\quad \nu\leq \alpha< 3\nu.$$ 
For that, we first perform a transformation
$$z^*=z\lambda(w),\quad w^*=w,\quad\lambda\in\RR{}\{w\},\,\lambda(0)=1$$
in such a way that for the new hypersurface $\im\Theta^*_{\nu\nu}=0$ (the argument is analogous to that in the case $\nu=1$). Then we note that, in view of \eqref{strong} and the implicit function procedure relating $\Phi$ and $\Theta$, we always have $\re\Theta^*_{\nu\nu}=0$ for a hypersurface given by \eqref{strong}, that is why we get in turn $\Theta^*_{\nu\nu}=0$. Second, we perform  a
sequence of $(k-1)$ transformation $F_2,...,F_{k}$, where each $F_j$ has the form
$$(z^*,w^*)=(z+f_j(w)z^j,w)$$
and achieve $\Theta^*_{\nu\alpha}=0,\quad \nu< \alpha< 3\nu.$

\medskip

\noindent \it Step II. \rm  We achieve the conditions 
$$\Theta^*_{\nu,3\nu}=\im\Theta^*_{2\nu,2\nu}=0$$ by performing a transformation
$$z^*=e^{ih(w)}z+s(w)z^{k+1},\,w^*=w,\quad h\in\RR{}\{w\},\,h(0)=0,\quad s\in\CC{}\{w\}.$$
We obtain a first order nonsingular ODE for $h(\bar w)$ which we solve uniquely with the initial condition $h(0)=0$. After that we uniquely determine the function $s(\bar w)$.

\medskip

\noindent \it Step III. \rm  We achieve the conditions 
$$\Theta^*_{\nu,\alpha}=0,\quad 3\nu< \alpha< 5\nu$$ by performing a 
sequence of $(k-1)$ transformation $F_2,...,F_{k}$, where each $F_j$ has the form
$$(z^*,w^*)=(z+f_j(w)z^{k+j},w).$$

\medskip

\noindent \it Step IV. \rm  We achieve the conditions 
$$\Theta^*_{\nu,5\nu}=\im\Theta^*_{3\nu,3\nu}=0$$ by performing a transformation
$$z^*=\lambda(w)z+p(w)z^{2k+1},\,w^*=g(w),\quad g\in\RR{}\{w\},\,g(0)=0,\,g\rq{}(0)\neq 0,\quad p\in\CC{}\{w\}$$
with
$$\lambda(w):=(g\rq{}(w))^{1/k}.
$$
We obtain in this step a third order ODE for the function $g(\bar w)$ which we solve uniquely with some initial conditions 
$$g(0)=0,\quad g\rq{}(0)=a\in\RR{}\setminus\{0\},\quad g\rq{}\rq{}(0)=b\in\RR{}.$$

 \medskip

\noindent \it Step V. \rm Finally, we perform a transformation
$$z^*=z+f(z,w),\quad w^*=w,\quad f(z,w)=O(|z|^{2k+2})$$
by which we  achieve the conditions $\Theta^*_{\nu\alpha}=0,\,\alpha>5\nu$. The new hypersurface then satisfies all the normalization conditions in \eqref{strongcirc}. Arguing as above we see 
in the circular case that we can canonically, up to the action of the  real projective group 
$$t\mapsto\frac{at}{1+bt},\,a,b\in\RR{},\,a\neq 0,$$ 
choose a parameter $t$ on each degenerate
chain.\rm

\medskip

\noindent \bf Case T2. \rm For a hypersurface of class T2 satisfying \eqref{strong} we first argue identically to the case T1 and perform a unique transformation $$z=z^*+f(z^*,w^*),\quad w=w^*, \quad f(z^*,w^*)=O(|z^*|^{2})$$ bringing $M$ to the form \eqref{strong1}. 
We claim then
that for the new hypersurface we have  $\re\Phi_{k,k-1}(u)=0$ in \eqref{strong1}. 

Indeed,
consider the (formal) transformation $F=(f,g)$ with $dF|_0=\mbox{Id}$, bringing a
hypersurface \eqref{strong1}  into normal form \eqref{normalformtube}, and study the
equation \eqref{tangency}, applied to it. Note that $M$ is in the normal form \eqref{normalformtube} up to weight $2k-2$. As follows from the normalization procedure, all (formal) maps bringing to a normal form up to weight $2k-2$ coincide with the identity up to weights $m-k+1$ in $f$ and $m$ in $g$, respectively. Let us use notations from the proof of Proposition 2.3. Then, using the above observation and collecting terms with 
 $z^{k-2}\bar z u^1$, we first obtain $\re f_0\rq{}(0)=0$. After that,  gathering terms with $z^{k-1}\bar z^0u^1,\,z^{2k-2}\bar z^1u^0,\,z^k\bar
z^{k-1}u^0$, it is
straightforward to check that we obtain, respectively, the equation \eqref{k-1,0}, differentiated and evaluated at $0$ with $\Psi_{k-1}\rq{}(0):=0$,  the equation \eqref{2k-2,1}, evaluated at $0$ with $\Psi_{2k-2,1}(0):=0$, and the equation \eqref{k,k-1}, evaluated at $0$ with $\Psi_{k,k-1}(0):=2\Phi_{k,k-1}(0)$. Moreover,
$\im f'_0(0)=0$ thanks to the fact that $\Gamma$ is a
degenerate chain. We immediately conclude from the above obtained equations that
$g_{k-1}\rq{}(0)=f_{k}(0)=0$ and
$\re\Phi_{k,k-1}(0)=0$. Since the prenormal form \eqref{strong1} is
invariant under the real shifts $w\mapsto w+u_0,\,u_0\in\RR{}$,
and $\Gamma\subset M$ is a degenerate chain, we similarly conclude
that $\re\Phi_{k,k-1}(u_0)= 0$ for any small $u_0\in\RR{}$ in
\eqref{strong1}, as required. Thus we have achieved all the normalization conditions in \eqref{normalformtube}, except $\im\Phi_{2k-2,2}(u)=0$. The step in the normalization procedure, aiming to achieve the normalization condition $\im\Phi_{2k-2,2}(u)=0$, can be addressed as a choice of a parameterization on the unique degenerate chain, passing through $0$, and is completely analogues to the similar step in the cases G and T1 with $\nu=1$, so that we leave the details   to the reader. In particular, in view of the fact that $M$ is already in the normal form \eqref{normalformtube},
we can canonically, up to the action of the group of dilations
\eqref{dilations}, choose a parametrization on each degenerate
chain.\rm

\subsection{Uniqueness properties for a normalizing transformation}  For the proof of the uniqueness statements in Theorems 1 and 2 we address each of the cases G,C,T1,T2 separately.

\medskip

\noindent \bf Case G. \rm Since a hypersurface of kind G, satisfying \eqref{stronggeneric}, is also in the normal form \eqref{normalformgeneric},  the only statement that needs to be proved in this case is the fact that there exists a unique degenerate chain, passing through $p$. As follows from the proof in Section 5.1, a degenerate chain is given by the equation \eqref{Gamma} in an appropriate normal form \eqref{stronggeneric} of $M$. Now the desired uniqueness follows from Corollary 2.4 and the fact that any transformation \eqref{dilations} preserves the curve \eqref{Gamma}. 

\medskip

\noindent \bf Case C. \rm We again notice that any hypersurface of kind C, which is in a a strong normal form \eqref{strongcirc}, is also in the normal form \eqref{normalformcirc}. This implies the required uniqueness statement for the normalizing transformation.
%Next, we show that any transformation \eqref{biggroup} preserves the strong normal form \eqref{normalformcirc}. As the statement is obvious for a transformation, belonging to the subgroup in \eqref{biggroup} normalized by $r=0$, it is sufficient to prove it for the subgroup in \eqref{biggroup}, normalized by $\lambda=1,\,\theta=0$. Each such transformation $F$ satisfies all the conditions in \eqref{gaugecirc}, except $g\rq{}\rq{}(0)=0$ (we have $g\rq{}\rq{}(0)=-2r$). Then, performing a calculation, identical to \eqref{findcirc}, we get that the fact that the new hypersurface $M^*$ satisfies \eqref{strongcirc} amounts to the third order nonsingular ODE 
%\begin{equation}\label{ODEformap}
%\frac{1}{3}g\rq{}\rq{}\rq{}-\frac{(g\rq{}\rq{})^2}{2g\rq{}}=0.
%\end{equation}
%It is easy to check that the function $g(w):=\frac{w}{1+rw}$ satisfies \eqref{ODEformap}, so that $M^*$ is in the strong normal form  \eqref{strongcirc}, as required. It is straightforward now that any two strong normal forms of a fixed germ of kind C are related by means of a transformation \eqref{biggroup}.
To prove the uniqueness of the degenerate chain passing through $p$, let us consider a hypersurface \eqref{strong} with a circular polynomial model. Then, in view of \eqref{typel}, the constant  type  locus of $M$ is contained in the set $E=\{d^{k-3}\Delta(q)|_{T^{\CC{}}_q}=0\}$ (here $\Delta$ is the Levi form). For a hypersurface \eqref{strong} it gives
$$E=\bigl\{z+O(2)=0\bigr\},$$
where $O(2)$ denotes terms of degree $\geq 2$ in $z,\bar z,u$. Thus $E$ is a smooth curve in $M$,
so that it coincides with $\Gamma$, as required for the uniqueness. 

\medskip

\noindent \bf Case T1. \rm In this case we shall use the special normal form, considered in Section 3.  Let $M$ be a hypersurface of class T1, which is brought into a strong normal form \eqref{stronggeneric} after a certain choice of a degenerate chain $\gamma$ through $p$, as described in the proof in Section 5.1. Then $M$ is also in the special normal form \eqref{specialtube}, associated with the triple $(M,\gamma,p)$. Now the uniqueness statement for a strong normalization map $F$ follows from Proposition 3.1.

If, for some choice of degenerate chains $\gamma_1 \ni p_1,\,\gamma_2\ni p_2$, two germs $(M_1,p_1),\,(M_2,p_2)$ of class T1 have the same strong normal form, then they are obviously biholomorphically equivalent. On the other hand, if two germs $(M_1,p_1),\,(M_2,p_2)$ of class T1 are biholomorphically equivalent by means of a mapping $H:\,(M_1,p_1)\mapsto (M_2,p_2)$, we fix a degenerate chain $\gamma_1\subset M_1$ through $p_1$. Then $\gamma_1$ is mapped into a degenerate chain $\gamma_2\subset M_2$ through $p_2$. Thus $H$ performs a biholomorphic equivalence between the triples 
$(M_1,\gamma_1,p_1)$ and $(M_2,\gamma_2,p_2)$ (in the sense of the considerations in Section 3), and by Corollary 3.2 some special normal forms of them are mapped onto each other by means of a transformation \eqref{dilations}, as required.

  \medskip

\noindent \bf Case T2. \rm As any hypersurface of kind T2, satisfying \eqref{strongtube2}, is also in the normal form \eqref{normalformtube}, the proof of all desired facts immediately follows from the formal normalization
Proposition 2.1 and the proof in Section 5.1.

\smallskip

Theorems 1 and 2 are completely proved now. The following example shows, that the discrete parameter $\sigma$, corresponding to a choice of a degenerate chain in $M$ through $p$ in the case T1, can actually occur. 

\begin{ex} 
Consider the real-analytic hypersurface $M\subset\CC{2}$ of type $4$ at the origin with a tubular polynomial model, given by the equation 
$$
v=(\re z)^3\bigl(\re z+(\im z)^2-u^2\bigr).$$
It is easy to check that at all points, belonging to either of the curves 
$$\gamma_\pm=\bigl\{\re z=0,\,\im z=\pm u,\,v=0\bigr\},$$ 
the hypersurface $M$ has type $4$, while at points with $z=0,\,v=0,\,u\neq 0$ the type of $M$ equals $3$. Hence the Levi degeneracy set $\Sigma$ has dimension $2$, and $M$ is of class T1 at $0$ (as follows from Propositions 4.1 and 4.2). At the same time, $M$ contains at least two transverse curves of constant type through $0$.
\end{ex}

\section{Canonical connection for hypersurfaces of class T2}

In this section we apply the strong normal form to construct a canonical connection on the Levi degeneracy set of a hypersurface of class T2, such that the restriction of its complexification onto the initial hypersurface $M$ gives a solution for the holomorphic equivalence problem for germs at Levi degenerate points of hypersurfaces of class T2.  Recall that the Levi degeneracy set $\Sigma$ in the case T2 is a totally real surface $\Sigma\subset\CC{2}$, and one can define the transverse and the tangent foliations $\mathcal T$ and $\mathcal S$ on $\Sigma$, respectively. Both  $\mathcal T$ and $\mathcal S$ are biholomorphically invariant and transverse to each other.

 Let $q\in\Sigma$ be a point in a neighborhood of the reference point $p$, and $X,Y$ denotes a pair of smooth vector fields on $\Sigma$ near $q$. We denote the leaves of the foliations $\mathcal T$ and $\mathcal S$ through $p$ by $t_p$ and $s_p$, respectively. We can assume (by fixing a tangent vector $e_p\in T_p t_p$) that all the chains $t_q$ are parameterized (see Section 5). In order to define a connection $\nabla_Y(X)$ on $\Sigma$, it is sufficient to define, for any $q\in\Sigma$ and each pair $X,Y$, a parallel translation of $X_q$  along the integral curve of $Y$ through $q$. Since $\mathcal T$ and $\mathcal S$ are transverse to each other, it is sufficient, by linearity, to define the parallel translation in the cases when (i)  $X_q\in T_q t_q, Y_q\in T_q s_q$;  (ii)  $X_q\in T_q s_q, Y_q\in T_q t_q$; (iii)  $X_q\in T_q t_q, Y_q\in T_q t_q$; (iv)  $X_q\in T_q s_q, Y_q\in T_q s_q$. In other word, we define a parallel translation for a vector, tangent to either of the leaves $t_q,s_q$, along either of the leaves $t_q,s_q$.

In case (i), we translate $X_q$ to a point $r\in s_p$ as follows. We consider the leaf $t_r$, and consider the analytic diffeomorphism $h$ of $(t_q,q)$ onto $(t_r,r)$, which assigns to each $\zeta\in t_q$ the intersection point $s_\zeta\cap t_q=:\{\zeta\rq{}\}$. Then the desired vector $Z$, parallel to $X_q$, is defined as
$Z:=  dh|_q(X_q)$.

Case (ii) is completely analogues to (i).

In case (iii), we use the existence of a parameterization $t_q=\bigl\{\gamma_q(\xi)\bigr\},\,\gamma_q(\xi_1)=q,\,\gamma_q(\xi_2)=r$ on the chain $t_q$ to simply translate the vector $\gamma_q\rq{}(t_1)$ into the vector $\gamma_q\rq{}(t_2)$, for any $r\in t_q$, and then by linearity extend the procedure for all other $X_p\in T_q t_q$. Clearly, this definition is independent of the choice of the vector $e_p$.

Finally, in case (iv) we note that a hypersurface of finite type $k\geq 3$ with a tubular polynomial model at a point $q\in M$ admits a canonical nondegenerate $k$-form
$$\varphi_q:\,(T^{\CC{}}_q M)^{k-1}\times \overline{T^{\CC{}}_q M}\mapsto\CC{}\otimes (T_q M/T^{\CC{}}_q M),$$
given by $k$-th order brackets (as can be seen from \eqref{tangentmodel}). Since $M$ has the same polynomial model at each point $q\in\Sigma$ (see Section 3), we get a canonical nondegenerate $k$-form $\varphi$ on $\Sigma$. Considering its restrictions on the tangent spaces $T_q s_q$, we obtain a canonical nondegenerate $k$-form $\psi$, which we may interpret as a one valued in $T_q t_q$ for each $q\in\Sigma$:
$$\psi_q:\,(T_q s_q)^{k-1}\times \overline{T_q s_q}\mapsto T_q t_q.$$
Now the desired parallel translation in case (iv) is performed as follows. Fixing a point $r\in s_q$, we take the vector $W:=\psi(X_q,...,X_q,\overline{X_q})$, translate in parallelly along the leaf $s_p$ as described in case (i), and obtain the resulting vector $V\in T_r t_r$. Now the desired vector $Z\in T_r s_r$ is uniquely defined by the condition $\psi(Z,...,Z,\bar Z)=V$.

Thus we obtain a unique canonical connection $\nabla$, invariant under biholomorphic transformations of a germ $(M,p)$. As follows from the construction, all the chains $t_q$, as well as the tangent leaves $s_q$, are geodesics for this connection. We may consider now the exponential mapping (geodesic flow) $$\mbox{exp}(X):\,(\RR{2},0)\mapsto (\Sigma,p)$$ for the connection $\nabla$, and the map $\mbox{exp}^{-1}$ gives canonical (normal) coordinates $(x_1,x_2)$ for $\Sigma$. If $(M,p)$ was initially simplified in such a way that
$$T_p t_p=\{z=0,\,v=0\},\quad T_p s_p=\{w=0,\re z=0\},$$
 then
 in the normal coordinates $p$ is the origin, the transverse leaves $t_q$ all look as $\{x_1=const\}$, and the tangent leaves $s_q$ as $\{x_2=const\}$. We can further fix the exponential mapping uniquely by the requirement that all transverse leaves (the chains) are canonically, up to a non-zero real scalar $\mu$, parameterized by a real parameter $\xi\in\RR{}$ as
$$\bigl\{x_1=const,\,x_2=\mu\xi\bigr\}.$$
In addition, it follows from the construction of the connection that the canonical $k$-form $\psi$ on  tangent leaves looks as $$\psi=\mu (x_1)^{k}$$
at any point $q\in\RR{2}$.

Considering now the complexification $\Sigma^{\CC{}}=\CC{2}$ of the totally real surface $\Sigma$, we obtain the compexification
$\nabla^{\CC{}}$
%:\, \CC{2}\times\CC{2}\mapsto\CC{2}$$
of the (real-analytic) connection $\nabla$ as a
holomorphic connection on $\Sigma^{\CC{}}$. 
The complexifications $\mathcal T^{\CC{}}$ and $\mathcal S^{\CC{}}$ of the foliations $\mathcal T$ and $\mathcal S$ respectively form a pair of foliations of $\CC{2}$, transverse to each other.  
%The leaves of these foliations are geodesics for $\nabla^{\CC{}}$. 
%Then the restriction
%$$\nabla^{\CC{}}_M:\, TM\times TM\mapsto\CC{2}$$
%of $\nabla^{\CC{}}$ onto $M$ gives a canonical object on $M$. 
The complexification $\mbox{exp}^{\CC{}}(Z)$ gives the exponential mapping for $\nabla^{\CC{}}$, and the map $(\mbox{exp}^{\CC{}})^{-1}$ provides canonical (normal) coordinates $(z_1,z_2)$ for $(M,p)$. 
%Similarly to the case of a real connection, the foliations $\mathcal T$ and $\mathcal S$ become, respectively, vertical and horizontal, the vertical leaves are all parameterized as
%$$\bigl\{z_1=const,\,z_2=\mu\xi\bigr\},$$
%and the complexification of the canonical $k$-form $\psi$ looks as
% $$\psi^{\CC{}}=\mu (z_1)^{k-1}\bar z_1$$
% at any point in $\CC{2}$.
 
 Now, in order to prove Theorem 3, it is sufficient to prove the following proposition.

\begin{propos}
 Let $(M,p)$ and $(M^*,p^*)$ be germs of two real-analytic hypersurfaces of the same type $k\geq 3$ as above, $\Sigma$ and $ \Sigma^*$ their Levi degeneracy sets, and $\nabla^{\CC{}}$ and $(\nabla^*)^{\CC{}}$
 the corresponding canonical complexified connections. Let $(N,0)$ and $(N^*,0)$ be two normal forms of $(M,p)$ and $(M^*,p^*)$, given by normal coordinates for the connections $\nabla$ and $\nabla^*$, respectively. Then any biholomorphism $F:\,(N,0)\mapsto (N^*,0)$ is a linear mapping, as in \eqref{dilations}.
 \end{propos}

 \begin{proof}
 Normal coordinates for the connections  $\nabla^{\CC{}}$ and $(\nabla^*)^{\CC{}}$, respectively, are defined up to a choice of parameterization on the vertical leaves $t_p$ (which, in turn, is reduced to a scalar $\mu\in\RR{},\,\mu\neq 0$, see Section 5). First notice that the map $F$ must preserve the leaves $\{z_1=const\},\,\{z_2=const\}$, thus it has the form 
 $$z_1\mapsto f(z_1),\,z_2\mapsto g(z_2).$$
 Further, $F$ preserves the parameterization on the vertical leaves, so that $g$ has the form 
 $$g(w)=\mu w,\,\mu\neq 0.$$ 
 Finally, $F$ must preserve the canonical form $\psi$, which is a real multiple of $(z_1)^{k-1}\bar z_1$ at each point, thus $f$ looks as 
 $$f(z_1)=(\mu)^{1/k}z_1.$$
 This prove the proposition.
 \end{proof}
Theorem 3 is  proved now.
%For the proof of Theorem 3, it is sufficient to show that if $(M,p)$ and $(M^*,p^*)$ are germs of two %hypersurfaces of the same type $k\geq 3$ and of class T2 at the points $p$ and $p^*$ respectively, %$\nabla$ and $\nabla^*$ are the canonical connections on their Levi degeneracy sets $\Sigma$ and $ %\Sigma^*$, respectively, and there is an analytic diffeomorphism $H:\,(\Sigma,p)\mapsto %(\Sigma^*,p^*)$, preserving the connections, then $(M,p)$ and $(M^*,p^*)$ are biholomorphic (in fact, %by means of the complexification of the mapping $H$). Indeed, transferring to normal coordinates for %$\Sigma,\Sigma^*$, as described above, we see that the map $H$ must preserve the leaves %$\{x_1=const\},\,\{x_2=const\}$, thus

\begin{flushleft}

\end{flushleft}

\end{document}